\documentclass{amsart}
\pdfoutput=1

\usepackage[utf8]{inputenc}
\usepackage{amsfonts,mathtools,amsmath,amsthm,amssymb,tikz,color,enumitem,latexsym,tikz-cd,mathrsfs,graphicx,verbatim,MnSymbol}
\usepackage[breaklinks]{hyperref}
\usepackage[nameinlink, capitalise]{cleveref}

\usepackage[backref=true,maxbibnames=99]{biblatex}
\bibliography{sources.bib}
\allowdisplaybreaks

\hypersetup{
	colorlinks,
	hyperfootnotes=true,
	linkcolor={red!50!black},
	citecolor={blue!50!black},
	urlcolor={blue!80!black}
}

\DefineBibliographyStrings{english}{%
	backrefpage = {page},
	backrefpages = {pages},
}

\usetikzlibrary{arrows}

\tikzset{slopearrow/.style={sloped, anchor=south}}

\newtheorem{lem}{Lemma}[section]

\newtheorem{thm}[lem]{Theorem}
\newtheorem{prop}[lem]{Proposition}

\newtheorem{cor}[lem]{Corollary}
\newtheorem{defn}[lem]{Definition}

\theoremstyle{definition}

\newtheorem{ex}[lem]{Example}
\newtheorem{rmk}[lem]{Remark}

\def\CC{\mathbb C}
\def\NN{\mathbb N}
\def\RR{\mathbb R}
\def\ZZ{\mathbb Z}
\def\QQ{\mathbb Q}
\def\FF{\mathbb F}

\def\k{k}

\DeclareMathOperator\GL{GL}

\DeclareMathOperator\SO{SO}

\DeclareMathOperator\disc{disc}
\DeclareMathOperator\Tr{Tr}

\DeclareMathOperator\End{End}

\DeclareMathOperator\ord{ord}

\DeclareFontFamily{U}{matha}{\hyphenchar\font45}
\DeclareFontShape{U}{matha}{m}{n}{
	<5> <6> <7> <8> <9> <10> gen * matha
	<10.95> matha10 <12> <14.4> <17.28> <20.74> <24.88> matha12
}{}
\DeclareSymbolFont{matha}{U}{matha}{m}{n}
\DeclareFontFamily{U}{mathx}{\hyphenchar\font45}
\DeclareFontShape{U}{mathx}{m}{n}{
	<5> <6> <7> <8> <9> <10>
	<10.95> <12> <14.4> <17.28> <20.74> <24.88>
	mathx10
}{}
\DeclareSymbolFont{mathx}{U}{mathx}{m}{n}

\DeclareMathSymbol{\obot}         {2}{matha}{"6B}
\DeclareMathSymbol{\bigobot}       {1}{mathx}{"CB}

\def\black{\color{black}}

\title[Elements of prescribed norm in maximal orders]{On elements of prescribed norm in maximal orders of a quaternion algebra}
\author{Eyal Z.~Goren and Jonathan Love}

\begin{document}
	
	\begin{abstract}
		Let~$\mathcal{O}$ be a maximal order in the quaternion algebra over~$\QQ$ ramified at~$p$ and~$\infty$. We prove two theorems that allow us to recover the structure of~$\mathcal{O}$ from limited information. The first says that for any infinite set~$S$ of integers coprime to~$p$,~$\mathcal{O}$ is spanned as a $\ZZ$-module by elements with norm in~$S$. The second says that~$\mathcal{O}$ is determined up to isomorphism by its theta function.
	\end{abstract}
	
	\maketitle
	
	\section{Introduction}
	
	Let~$p$ be a prime. Up to isomorphism there is a unique quaternion algebra~$B_p$ over~$\QQ$  ramified at exactly~$p$ and~$\infty$. The quaternion algebra~$B_p$ comes equipped with a canonical involution $x\mapsto \overline{x}$, a norm $x\mapsto N(x)\coloneqq x\overline{x}$, and a trace $x\mapsto \Tr(x)\coloneqq x+\overline{x}$. 
	
	There will typically be many non-isomorphic maximal orders in~$B_p$: the number of isomorphism classes of maximal orders in~$B_p$ (the \emph{type number} of a maximal order)
	is between~$\frac{p-1}{24}$ and~$\frac{p+13}{12}$ inclusive \cite[Exercise 30.6 and Proposition 30.9.2]{voight}.
	This paper presents two theorems, each of which allows one to recover information about a maximal order~$\mathcal{O}$ in~$B_p$, given only information about elements in~$\mathcal{O}$ with prescribed norms.
	
	\begin{thm}\label{thm:genbylpower}
		Let~$\mathcal{O}$ be a maximal order in~$B_p$, and let~$S$ be an infinite set of positive integers coprime to~$p$. There is a generating set for~$\mathcal{O}$ as a~$\ZZ$-module consisting of elements with norm in~$S$.
	\end{thm}
	\begin{rmk}
		See \cref{rmk:coprimetop} for a discussion of the coprime to~$p$ condition. 
		The theorem still holds if we take~$\mathcal{O}$ to be an Eichler order in~$B_p$ with index coprime to~$6$, but is false for every other Eichler order in~$B_p$; see \cref{sec:eichler}. 
	\end{rmk}

	To prove \cref{thm:genbylpower}, we first establish a local-global principle for lattices having the property of being generated as a~$\ZZ$-module by elements of norm in a given set~$S$ (\cref{thm:localglobal}). In \cref{sec:maxorder_localgens}, we check that under the conditions of \cref{thm:genbylpower}, all the local conditions of \cref{thm:localglobal} are satisfied.

	As a special case, we can take $S=\{\ell^k:k\geq 0\}$ for any prime $\ell\neq p$, and conclude that maximal orders are generated by elements of norm equal to a power of~$\ell$.  This has implications for the study of isogeny graphs of supersingular elliptic curves. If~$E$ is a supersingular elliptic curve over~$\overline{\FF_p}$, then~$\End(E)$ is a maximal order in~$B_p$, and endomorphisms of~$E$ with norm a power of~$\ell$ can be generated by finding cycles from~$E$ in the~$\ell$-isogeny graph of supersingular curves over~$\overline{\FF_p}$ (see for instance \cite{kohel,eisentrageretal,banketal}). \cref{thm:genbylpower} implies that such endomorphisms generate the entire endomorphism ring as an abelian group.  This is used as a heuristic assumption in~\cite[Section 3.3]{eisentrageretal}.  This question is considered by \cite{banketal}, who determine when two cycles generate a full rank sublattice of~$\End(E)$ (as an order), as well as a necessary condition for these cycles to generate~$\End(E)$, but show by example that these necessary conditions are not sufficient.
	
	The second result is that the isomorphism type of~$\mathcal{O}$ is determined by the number of elements of each norm. Given a lattice~$\Lambda$ with integral quadratic form~$Q$, we define the \emph{theta function} of~$\Lambda$,
	\[\theta_\Lambda(q)\coloneqq \sum_{x\in\Lambda} q^{Q(x)},\]
	so that the coefficient of~$q^n$ is the number of elements of norm~$n$. This function encodes the spectrum of the Laplace operator of the Riemannian manifold $\Lambda\otimes\RR/\Lambda$ (see \cite{nilssonrowlettrydell} for more on this analytic interpretation). When~$\mathcal{O}$ is a maximal order in~$B_p$,~$\theta_{\mathcal{O}}(q)$ is a modular form of weight~$2$ and level~$\Gamma_0(p)$.
	
	We say two lattices are \emph{isospectral} if their theta functions are equal. Lattices of rank~$n\leq 3$ are uniquely determined up to isometry by their theta function \cite{schiemann}, but in rank~$4$ and above there exist pairs of non-isometric isospectral lattices. Even if we restrict to the set of lattices in the genus of a fixed maximal order of~$B_p$, there may exist pairs of non-isometric integral lattices in~$B_p$ whose left- and right-orders are maximal and yet have the same theta function; see \cref{sec:isospectral} for examples. However, we prove that this does not occur if we restrict to maximal orders in~$B_p$.
	\begin{thm}\label{thm:theta_to_iso}
		If~$\mathcal{O},\mathcal{O}'$ are maximal orders in~$B_p$ with $\theta_\mathcal{O}=\theta_{\mathcal{O}'}$, then $\mathcal{O}\simeq \mathcal{O}'$.
	\end{thm}
	
	The proof of \cref{thm:theta_to_iso} is divided into two steps; the first may be of independent interest so we state it here as a separate theorem. Given an order~$\mathcal{O}$ in~$B_p$, we define its \emph{Gross lattice} 
	\[\mathcal{O}^T\coloneqq \{2x-\Tr(x):x\in\mathcal{O}\}.\]
	This is a strict subset of the set of trace~$0$ elements in~$\mathcal{O}$; see \cref{sec:gross_lattice} for further details and discussion of the Gross lattice. For~$i=1,2,3$, the~$i$-th \emph{successive minimum} of~$\mathcal{O}^T$ is the minimum value~$D_i$ such that the span of all elements~$\alpha\in\mathcal{O}^T$ with~$N(\alpha)\leq D_i$ has dimension at least~$i$ (\cref{def:succmin}). 	
	
	In \cite{chevyrev_galbraith}, Chevyrev and Galbraith determine conditions under which the successive minima of the Gross lattice of~$\mathcal{O}$ determine the isomorphism class of~$\mathcal{O}$. The following result is a strengthening of \cite[Theorem 1]{chevyrev_galbraith}, and uses many of the same methods.
	
	\begin{thm}\label{thm:succmin_to_iso}
		Let~$p$ be an odd prime. Suppose~$\mathcal{O}_1$,~$\mathcal{O}_2$ are orders of~$B_p$, each of index~$r$ in some (not necessarily the same) maximal order. Suppose~$\mathcal{O}_1^T$ and~$\mathcal{O}_2^T$ have the same successive minima $D_1\leq D_2\leq D_3$, and that  $D_1\geq 8r^2$.  Then $\mathcal{O}_1\simeq\mathcal{O}_2$.
	\end{thm}
	
	In particular, for all primes~$p$, a maximal order in~$B_p$ is determined up to isomorphism by the successive minima of its Gross lattice:  this is vacuously true if $p=2$ or $D_1<8$ because this information determines a unique maximal order in~$B_p$ (see \cref{lem:smallcases}), and it holds for $D_1\geq 8$ by \cref{thm:succmin_to_iso}. 
	
	After setting up some preliminary results on the geometry of quaternion orders in \cref{sec:quat_background}, we prove \cref{thm:succmin_to_iso} in \cref{sec:mins_to_order}. The remainder of \cref{sec:thetaproof} is used to show that the theta function of~$\mathcal{O}$ determines the successive minima of~$\mathcal{O}^T$, allowing us to conclude \cref{thm:theta_to_iso}.

	In future work we will explore similar questions for quaternion algebras over totally real fields.

	\subsection{Acknowledgements}
	
	We thank Eran Assaf, Naser Sardari, and John Voight for useful discussion.
	
	\subsection{Lattice definitions and conventions}\label{sec:notation}
	Let $R=\ZZ$ or $R=\ZZ_\ell$ for some prime~$\ell$, and~$K$ the fraction field of~$R$.  A \emph{lattice}~$\Lambda$ is a free finite-rank~$R$-module (so that $\Lambda\simeq R^n$ for some positive integer~$n$) equipped with a non-degenerate quadratic form $Q\colon \Lambda\to K$. \black A quadratic form $Q$ is \emph{integral} if it takes values in~$R$. For $\mathbf{x}\in \Lambda$, we will refer to the value~$Q(\mathbf{x})$ as the \emph{norm} of~$\mathbf{x}$. Any quadratic form defines a bilinear form $x\cdot y\coloneqq \frac12(Q(x+y)-Q(x)-Q(y))$; if~$Q$ is integral than the bilinear form is valued in~$\frac12 R$. Given a basis $\mathbf{v}_1,\ldots,\mathbf{v}_n$ for~$\Lambda$,  the \emph{Gram matrix for~$\Lambda$} (we will also say ``the Gram matrix of~$Q$'') \black is the symmetric matrix $\mathbf{A}_{\Lambda}\in \frac12 M_n(R)$ defined by 
	\[\mathbf{A}_{\Lambda}\coloneqq (v_i\cdot v_j)_{1\leq i,j\leq n}.\]
	If we write $\mathbf{x}\in\Lambda$ as a vector in terms of the basis $\mathbf{v}_1,\ldots,\mathbf{v}_n$, then the Gram matrix satisfies the relation 
	\[Q(\mathbf{x})=\mathbf{x}^T\mathbf{A}_{\Lambda}\mathbf{x}.\]
	If there is no room for confusion, the subscript of~$\mathbf{A}_\Lambda$ may be dropped. The \emph{determinant} of~$\Lambda$, $\det\Lambda$, is defined to be the determinant of a Gram matrix for~$\Lambda$. When $R=\ZZ$, we have $\det\Lambda>0$ and the value does not depend on the choice of basis.
	
	Given $a_1,\ldots,a_k\in \Lambda$, we use the notation $\langle a_1,\ldots,a_k\rangle$ to denote the sublattice of~$\Lambda$ generated by $a_1,\ldots,a_k$ as an $R$-module. We say that a subset $C\subseteq \Lambda$ is a \emph{generating set} for~$\Lambda$ if~$C$ generates~$\Lambda$ as an~$R$-module.

	\section{Generating sets for maximal orders}
	
	\subsection{A local-global principle for being generated by elements of prescribed norms}
	
	Let $Q\colon \ZZ^n\to\ZZ$ be a quadratic form with Gram matrix $\mathbf{A}\in \frac12 M_n(\ZZ)$.  For a prime~$\ell$, set 
	\[\tau_\ell=\left\{\begin{array}{ll}
		1,&\ell>2,\\
		3,&\ell=2.
	\end{array}\right.\]
	Aside from the last line, \black the following definition appears in Browning and Dietmann \cite{browning_dietmann}.
	\begin{defn}
		For $s\in\ZZ_{>0}$ and $\mathbf{A}\in M_n(\ZZ)$, the pair $(s,Q)$ satisfies the \emph{strong local solubility condition} (``strong~LSC'') if for every prime~$\ell$ there exists $\mathbf{x}\in (\ZZ/\ell^{\tau_\ell}\ZZ)^n$ with $Q(\mathbf{x})\equiv s\pmod {\ell^{\tau_\ell}}$ and $\ell\nmid \mathbf{A}\mathbf{x}$. 
		
		If $\mathbf{A}\in \frac12M_n(\ZZ)\setminus M_n(\ZZ)$, then we say $(s,Q)$ satisfies strong~LSC if $(2s,2Q)$ does.
	\end{defn}
	
	Note that the strong~LSC condition does not depend on the basis for $\ZZ^n$ used to define the Gram matrix $\mathbf{A}$. \black
	
	If for every prime~$\ell$ there exists $\mathbf{x}\in\ZZ_\ell^n$ with $Q(\mathbf{x})=s$, we say $(s,Q)$ satisfies the \emph{weak local solubility condition} (``weak~LSC''). Strong~LSC implies weak~LSC by Hensel-lifting, but the converse does not hold (\cref{ex:stronglsc}).
	
	The following theorem is a local-global principle for lattices with the property of being generated by elements with norm in~$S$. 
	
	\begin{thm}\label{thm:localglobal}
		Let $n\geq 4$,~$Q$ a non-degenerate integral quadratic form on~$\ZZ^n$, and $S\subseteq\ZZ_{>0}$. Suppose that for all $M\geq 0$ and all primes~$\ell$ there exists a generating set~$C_\ell$ for~$\ZZ_\ell^n$ such that for all $\mathbf{x}\in C_\ell$, the norm $s\coloneqq Q(\mathbf{x})$ satisfies
		\begin{enumerate}[label=(\alph*)]
			\item $s\in S$, 
			\item $s\geq M$, 
			\item $(s,Q)$ satisfies strong~LSC.
		\end{enumerate} 
		Then~$\ZZ^n$ has a generating set consisting of elements with norm in~$S$.
	\end{thm}
	
	We prove this in \cref{sec:localglobalproof}. Before that we draw a corollary, and then discuss the necessity of the conditions in the theorem.
	
	A quadratic form $Q\colon \ZZ^n\to \ZZ$ is \emph{primitive} if $\gcd(\{Q(\mathbf{x}):\mathbf{x}\in\ZZ^n\})=1$.
	
	\begin{cor}\label{cor:basislocalglobal}
		Let $n\geq 4$,~$Q$ a non-degenerate primitive integral quadratic form on~$\ZZ^n$, and $S\subseteq\ZZ_{>0}$ an infinite set. Suppose that for all $s\in S$ and that for all primes~$\ell$, there exists a basis for~$\ZZ_\ell^n$ consisting of elements of norm~$s$. Then~$\ZZ^n$ has a generating set consisting of elements with norm in~$S$.
	\end{cor}
	\begin{proof}
		Since~$s$ can be arbitrarily large, we just need to check that $(s,Q)$ satisfies strong~LSC. We have a basis consisting of elements $\mathbf{x}\in\ZZ_\ell^n$ with $Q(\mathbf{x})=s$, so it suffices to show that one such basis vector has $\ell\nmid \mathbf{A}\mathbf{x}$ (or $\ell\nmid (2\mathbf{A})\mathbf{x}$ when $\mathbf{A}\notin M_n(\ZZ)$). 
		
		Let~$\ell$ be an odd prime, and for the sake of contradiction, suppose $\ell\mid \mathbf{A}\mathbf{x}$ for all~$\mathbf{x}$ in a basis for~$\ZZ_\ell^n$. Then $\ell\mid \mathbf{A}\mathbf{x}$ for all $\mathbf{x}\in\ZZ_\ell^n$, so \[\ell\mid \mathbf{x}^t \mathbf{A}\mathbf{x}=Q(\mathbf{x}),\]
		contradicting the assumption that~$Q$ is primitive. Thus $(s,Q)$ must satisfy strong~LSC at~$\ell$.
		
		Now suppose $\ell=2$. If $\mathbf{A}\in M_n(\ZZ)$, the same argument as above applies. Now suppose $\mathbf{A}\in \frac12 M_n(\ZZ)\setminus M_n(\ZZ)$, and for the sake of contradiction suppose $2\mid (2\mathbf{A})\mathbf{x}$ for all~$\mathbf{x}$ in a basis for~$\ZZ_2^n$. Letting $\mathbf{B}\in \GL_n(\ZZ_2)$ denote the matrix with columns corresponding to this basis, we have $(2\mathbf{A})\mathbf{B}\in 2M_n(\ZZ_2)$. This implies $\mathbf{A}\mathbf{B}\in M_n(\ZZ_2)$, so multiplying on the right by $\mathbf{B}^{-1}\in \GL_n(\ZZ_2)$, we have $\mathbf{A}\in M_n(\ZZ_2)$. This contradicts our initial assumption on~$\mathbf{A}$, so $(2s,2Q)$ --- and therefore also $(s,Q)$ --- satisfies strong~LSC at~$2$.
	\end{proof}

	We now discuss the technical conditions in the statement of \cref{thm:localglobal}, and demonstrate through example that they cannot be removed or substantially weakened.

	\begin{rmk}
		The conditions $n\geq 4$, (b), and (c) of \cref{thm:localglobal} will be familiar to experts in the study of representability of integers by quadratic forms. When $Q\colon\ZZ^n\to \ZZ$ is a quadratic form with $n\geq 4$, Browning and Dietmann find an explicit lower bound~$M$ such that whenever $k\geq M$ and $(k,Q)$ satisfies strong~LSC,~$k$ is representable by~$Q$ \cite[Theorem 5]{browning_dietmann}. In \cite[Section 1.2]{browning_dietmann}, they discuss examples due to Watson \cite[Section 7.7]{watson} demonstrating that a local-global principle can fail if $k<M$ or if $(k,Q)$ does not satisfy strong~LSC.
		
		In each of the counterexamples below, on the other hand, every value in~$S$ is globally represented by~$Q$. Even in this setting we show that if we drop any of the conditions $n\geq 4$, (b), or (c), the existence of generating sets for~$\ZZ_\ell^n$ with norms in~$S$ for all primes~$\ell$ is not sufficient to conclude the existence of a generating set for~$\ZZ^n$ with norms in~$S$.
	\end{rmk}

	\begin{ex}
		If we remove the condition $n\geq 4$ from \cref{thm:localglobal}, a counterexample is given by the quadratic form 
		\[Q(x,y)=x^2+21y^2\]
		on~$\ZZ^2$ and $S=\{19^{2k}:k\geq 0\}$. The elements in~$\ZZ^2$ with norm in~$S$ generate an index~$4$ sublattice of~$\ZZ^2$: using the observation that $5^2+21\cdot 4^2=19^2$, we can factor each side of the equation $x^2+21y^2=19^{2k}$ into prime ideals of $\ZZ[\sqrt{-21}]$ to show that we must have $4\mid y$. But for any $k\geq 0$ and any prime~$\ell$, there is a basis for~$\ZZ_\ell^2$ consisting of elements of norm~$19^{2k}$: we have
		\begin{align*}
			\begin{array}{lclll}
				Q(19^k,0) & \equiv & Q(19^k,1)&\equiv 19^{2k}\pmod \ell & \text{for } \ell=3,7,\\
				Q(1,0) &\equiv & Q(2,1)&\equiv 19^{2k}\pmod 8, &\text{and}\\
				Q(6,1)&\equiv & Q(6,-1)&\equiv 19^{2k}\pmod {19},
			\end{array}
		\end{align*}
		and for remaining~$\ell$, we can use the fact that $x^2+21y^2-19^kt^2=0$ defines a smooth projective conic over~$\FF_\ell$ to find two independent points $(x,y)\in\FF_\ell^2$ with $Q(x,y)\equiv 19^k\pmod \ell$. In each case these Hensel-lift to a basis for~$\ZZ_\ell^2$ of elements with norm~$19^k$.
		
		We do not currently know whether or not it is sufficient to assume $n\geq 3$ in \cref{thm:localglobal}.
	\end{ex}
	
	\begin{ex}
		Condition (b) of \cref{thm:localglobal} (that the local generators have norm at least~$M$) can be thought of as a constraint coming from the infinite place of~$\QQ$. If we remove it, a counterexample is given by the quadratic form
		\[Q(x,y,z,w)=x^2+9y^2+9z^2+9w^2\]
		on~$\ZZ^4$ with $S=\{37\}$. The only $x\in\ZZ$ satisfying $x^2\equiv 37\pmod 9$ and $x^2\leq 37$ is $x=\pm 1$,  so the set of vectors of norm~$37$ generate an index~$16$ sublattice of~$\ZZ^4$ with basis
		\[(1,2,0,0),(-1,2,0,0),(1,0,2,0),(1,0,0,2).\]
		
		However, for all primes~$\ell$, there is a basis for~$\ZZ_\ell^4$ consisting of elements of norm~$37$:  for odd~$\ell$ we can use the above basis because~$16$ is a unit in~$\ZZ_\ell^\times$,  and for $\ell=2$ we can let~$u$ be a square root of~$\frac{11}{3}$ in~$\ZZ_2^\times$ and use
		\[(1,2,0,0),(2,u,0,0),(2,0,u,0),(2,0,0,u).\]
	\end{ex}
	
	\begin{ex}\label{ex:stronglsc}
		If we weaken condition (c) to merely requiring that  $(s,Q)$ satisfies weak~LSC,  a counterexample is given by the quadratic form
		\[Q(x,y,z,w)=3x^2+5y^2+11\cdot 15^2z^2+11\cdot 15^3w^2\]
		on~$\ZZ^4$ and $S=\{3^k:k\geq 1\text{ odd}\}\cup\{5^k:k\geq 1\text{ odd}\}$. The elements $(3^{(k-1)/2},0,0,0)$ and $(0,5^{(k-1)/2},0,0)$ show that every element of~$S$ is globally represented (and hence everywhere locally represented). Further, for any odd $k\geq 1$ and $\ell\neq 5$,~$\ZZ_\ell^4$ has a basis of elements of norm~$5^k$, using the observations that
		\begin{align*}
			Q(1,1,0,1) \equiv  Q(0,1,1,1) \equiv  Q(0,1,0,0) \equiv  Q(0,0,0,1) &\equiv  5^k\pmod{8},\\
			Q(0,1,0,0) \equiv  Q(1,1,0,0) \equiv  Q(0,1,1,0) \equiv  Q(0,1,0,1) &\equiv  5^k\pmod{3},\\
			Q(3r,0,0,0) \equiv  Q(0,r,0,0) \equiv  Q(0,r,1,0) \equiv  Q(0,r,0,1) &\equiv  5^k\pmod{11}
		\end{align*}
		with $r=5^{(k-1)/2}$. In a similar way we can show that for any odd $k\geq 1$ and $\ell\neq 3$,~$\ZZ_\ell^4$ has a basis of elements of norm~$3^k$.
		
		However, for $k\geq 4$, the only elements of norm~$5^k$ in~$\ZZ_5^4$ are in~$5\ZZ_5^4$:
		\begin{align*}
			0\equiv Q(x,y,z,w)&\equiv 3x^2\pmod 5  & &\Rightarrow 5\mid x;\\
			0\equiv Q(5x_1,y,z,w)&\equiv 5y^2\pmod{25} & &\Rightarrow 5\mid y;\\
			0\equiv Q(5x_1,5y_1,z,w)&\equiv 25(3x_1^2+4z^2)\pmod{125} & &\Rightarrow 5\mid x_1,z;\\
			0\equiv Q(25x_2,5y_1,5z_1,w)&\equiv 125(y_1^2+2w^2)\pmod{625} & &\Rightarrow 5\mid y_1,w.
		\end{align*}
		Thus $(5^k,Q)$ does not satisfy strong~LSC at~$5$. In a similar way, for $k\geq 4$ the only elements of norm~$3^k$ in~$\ZZ_3^4$ are in~$3\ZZ_3^4$, so $(3^k,Q)$ does not satisfy strong~LSC at~$3$. In particular, every element of~$\ZZ^4$ with norm in~$S$ satisfies $15\mid z,w$ (using the argument above for $k\geq 4$ and checking explicitly for small~$k$), so such elements do not generate~$\ZZ^4$.
	\end{ex}

	\subsection{Proof of \cref{thm:localglobal}}\label{sec:localglobalproof}
	
	The proof is a direct application of a strong approximation theorem of Sardari. We quote a special case of this theorem here.
	
	Given an integer~$s$, a prime~$\ell$, an integer $t_\ell\geq 0$, and $\mathbf{a}_\ell\in \ZZ_\ell^n$, define the local density
	\[\sigma_\ell(\mathbf{a}_\ell,t_\ell,s)\coloneqq \lim_{k\to\infty}\frac{n(\ell^k)}{\ell^{(n-1)k}},\]
	where
	\[n(\ell^k)\coloneqq \#\{\mathbf{x}\in(\ZZ/\ell^{k+t_\ell}\ZZ)^n:Q(\mathbf{x})\equiv s\pmod {\ell^{k+t_\ell}},\quad\mathbf{x}\equiv \mathbf{a}_\ell\pmod{\ell^{t_\ell}}\}.\]
	Given a choice of~$t_\ell$ and~$\mathbf{a}_\ell$ for all~$\ell$ with $t_\ell=0$ for all but finitely many~$\ell$, set $\mathfrak{S}(s)\coloneqq \prod_{\ell}\sigma_\ell(\mathbf{a}_\ell,t_\ell,s)$ and $V=\prod_{\ell} \ell^{-t_\ell}$.
	\begin{thm}[{\cite[Theorem 1.6]{sardari}\footnote{The full theorem has a stronger bound when $n\geq 5$, and includes terms accounting for archimedean constraints.}}]\label{thm:sardari}
		Let $n\geq 4$, $Q\colon \ZZ^n\to \ZZ$ a non-degenerate quadratic form, and $\epsilon>0$. For all primes~$\ell$ and all integers~$s$, the number of $\mathbf{x}\in\ZZ^n$ satisfying $Q(\mathbf{x})=s$ and $\mathbf{x}\equiv \mathbf{a}_\ell\pmod{\ell^{t_\ell}}$ for all primes~$\ell$ is
		\[\gg \mathfrak{S}(s)V^{n-1}s^{\frac{n-2}{2}}\left(1+O(V^{-3(n-3)/2}s^{\epsilon-\frac{n-3}{4}})\right),\]
		where the implied constants in~$\gg$ and~$O$ depends only on~$\epsilon$ and~$Q$.		
	\end{thm}

	\begin{lem}\label{lem:singseries}
		Let~$\ell$ be a prime,~$s$ an integer, and $\mathbf{a}_\ell\in\ZZ_\ell^n$ satisfying $Q(\mathbf{a}_\ell)=s$. Let $t_\ell=1$ and $t_{\ell'}=0$ for $\ell'\neq \ell$. Suppose $(s,Q)$ satisfies strong~LSC. Then for all $\delta>0$ we have
		$\mathfrak{S}(s)\gg |s|^{-\delta}$, where the implicit constant depends only on~$Q$,~$\ell$, and~$\delta$ (not on~$s$).
	\end{lem}
	\begin{proof}
		Consider the modified singular series
		\[\mathfrak{S}'(s,Q)\coloneqq \prod_{\ell'\text{ prime}}\lim_{k\to\infty} \frac{\#\{\mathbf{x}\in(\ZZ/{\ell'}^k\ZZ)^n:Q(\mathbf{x})\equiv s\pmod {{\ell'}^{k}}\}}{\ell^{(n-1)k}}.\]
		Note that for $\ell'\neq \ell$, the term at~$\ell'$ is equal to the term at~$\ell'$ of~$\mathfrak{S}(s)$. The terms at~$\ell$ are each bounded above and below by nonzero constants in~$s$ that depend on~$\ell$ (the lower bound follows by an application of Hensel's lemma), so~$\mathfrak{S}'(s)$ and~$\mathfrak{S}(s)$ have the same rate of decay in~$s$.
		
		Browning and Dietmann prove~\cite[Proposition 2]{browning_dietmann} that for any $\delta>0$, if $\mathbf{A}\in M_n(\ZZ)$ and $(s,Q)$ satisfies strong~LSC, then
		\[\mathfrak{S}'(s,Q)\gg |s\Delta_Q|^{-\delta},\]
		with~$\Delta_Q$ the discriminant of~$Q$, and the implicit constant depending only on~$\delta$. 
		
		If $\mathbf{A}\notin M_n(\ZZ)$, then $(s,Q)$ satisfying strong~LSC implies
		\[\mathfrak{S}'(2s,2Q)\gg |2s\Delta_{2Q}|^{-\delta}.\]
		Now $\mathfrak{S}'(2s,2Q)$ and $\mathfrak{S}'(s,Q)$ are the same at every prime except~$2$, where they differ by at most a constant factor. So we reach the same conclusion for $\mathfrak{S}'(s,Q)$.		
	\end{proof}
	
	\begin{proof}[Proof of \cref{thm:localglobal}]
		Fix any prime~$\ell$, and let~$M$ be large (we will specify how large later). Let~$C_\ell$ be a generating set for~$\ZZ_\ell^n$ satisfying condtions (a) through (c), let $\mathbf{a}_\ell\in C_\ell$, and let $s\coloneqq Q(\mathbf{a}_\ell)$. Set $t_\ell=1$ and $t_{\ell'}=0$ for all primes $\ell'\neq \ell$. By \cref{lem:singseries}, the corresponding singular series is asymptotically larger than $s^{-\frac{n-2}{2}}$. So by \cref{thm:sardari}, if~$M$ is sufficiently large, there exists $\mathbf{y}\in\ZZ^n$ satisfying $Q(\mathbf{y})=s$ and $\mathbf{y}\equiv \mathbf{a}_\ell\pmod{\ell}$. Here ``sufficiently large'' may depend on~$\ell$,~$Q$, and a choice of $0<\epsilon<\frac{n-3}{4}$, but these choices can all be made at the outset.
		
		Thus there is a set $\widehat{C_\ell}\subseteq \ZZ^n$ such that for each $\mathbf{a}_\ell\in C_\ell$, there is a corresponding $\mathbf{y}\in\widehat{C_\ell}$ with $Q(\mathbf{y})=Q(\mathbf{a}_\ell)\in S$ and $\mathbf{y}\equiv\mathbf{a}_\ell\pmod \ell$. Since the inclusion $\ZZ^n\to\ZZ_\ell^n$ induces an isomorphism $\ZZ^n/\ell\ZZ^n\to\ZZ_\ell^n/\ell\ZZ_\ell^n$, and~$C_\ell$ generates~$\ZZ_\ell^n$,~$\widehat{C_\ell}$ generates~$\ZZ^n/\ell\ZZ^n$. 
		
		Since elements in~$\ZZ^n$ with norm in~$S$ generate~$\ZZ^n/\ell\ZZ^n$ for all primes~$\ell$, we can conclude that such elements generate~$\ZZ^n$.
	\end{proof}
	
	\begin{rmk}
		If every element of~$S$ is coprime to~$2p$, then a much shorter proof of \cref{thm:localglobal} can be given using Theorem 1.2 of \cite{sardari}. The authors were informed by Naser Sardari that a correction needs to be made to this result: as written it only requires~$N$ to be odd, but in fact~$N$ must also be relatively prime to the discriminant of~$Q$.
	\end{rmk}

	\subsection{Local generating sets for maximal orders}\label{sec:maxorder_localgens}

	Let~$\mathcal{O}$ be a maximal order in~$B_p$ and let~$s$ be any positive integer relatively prime to~$p$. We will show that for all primes~$\ell$, $\mathcal{O}\otimes\ZZ_\ell$ has a basis consisting of elements of norm~$s$, so that \cref{thm:genbylpower} follows from \cref{cor:basislocalglobal}.
	
	\subsubsection{$\ell\neq p$}\label{sec:splitprimes}
	In this case $\mathcal{O}\otimes \ZZ_\ell\simeq M_2(\ZZ_\ell)$, with the norm on~$\mathcal{O}$ inducing the determinant on $M_2(\ZZ_\ell)$. The elements
	\[\begin{pmatrix}
		s&0\\0&1
	\end{pmatrix},\quad \begin{pmatrix}
		s&1\\0&1
	\end{pmatrix},\quad \begin{pmatrix}
		s&0\\1&1
	\end{pmatrix},\quad \begin{pmatrix}
		1&-1\\ s&0
	\end{pmatrix}\]
	each have norm~$s$, and these evidently span $M_2(\ZZ_\ell)$.
	
	\subsubsection{$\ell=p\neq 2$}		
	Let $K/\QQ_p$ be the unique unramified quadratic extension, with Galois group generated by~$\sigma$. Then
	\[B_p\otimes \ZZ_p\simeq \left\{\begin{pmatrix}
		u&p v\\ \sigma(v)&\sigma(u)
	\end{pmatrix}:u,v\in K\right\}\subseteq M_2(K)\]
	\cite[Corollary 13.3.14]{voight}, and $\mathcal{O}\otimes\ZZ_p$ is the corresponding valuation ring~\cite[Proposition 13.3.4]{voight}, obtained by restricting~$u$ and~$v$ to be in the valuation ring of~$K$. The norm on $\mathcal{O}\otimes\ZZ_p$ is $u\sigma(u)-p v\sigma(v)$.

	\begin{rmk}\label{rmk:coprimetop}
		If~$u$ is not a multiple of~$p$, then $u\sigma(u)-p v\sigma(v)$ is not a multiple of~$p$. This shows that any basis of a maximal order $\mathcal{O}\subseteq B_p$ must contain at least two elements with norm relatively prime to~$p$.
	\end{rmk}

	If $p\neq 2$ we have $K\simeq \QQ_p(\sqrt{a})$ for some $a\in \QQ_p$ such that~$a$ is not a square modulo~$p$. Therefore $\mathcal{O}\otimes\ZZ_p\simeq \ZZ_p^4$ with quadratic form 
	\[(x,y,z,w)\mapsto Q(x,y,z,w)\coloneqq x^2-ay^2-pz^2+apw^2.\] 
	Since $p\nmid s$, the equation $x^2-ay^2-st^2=0$ defines a smooth projective conic over~$\FF_p$. This curve has $p+1\geq 3$ $\FF_p$-points, none of which has $t=0$, and no three of which lie on a common line. Thus there exist two linearly independent points $(c_1,d_1),(c_2,d_2)\in\FF_p^2$ with $c_1^2-ad_1^2=c_2^2-ad_2^2=s$ in~$\FF_p$, so we have
	\[Q(c_1,d_1,0,0)\equiv Q(c_1,d_1,1,0)\equiv Q(c_1,d_1,0,1)\equiv Q(c_2,d_2,0,0)\equiv s\pmod p.\]
	By Hensel lifting, we obtain a basis for~$\ZZ_p^4$ consisting of elements of norm~$s$.
	
	\subsubsection{$\ell=p=2$}
	As above, the norm on $\mathcal{O}\otimes\ZZ_p$ is $u\sigma(u)-p v\sigma(v)$, but this time we have $K\simeq \QQ_2(\zeta_3)$, where $\zeta_3\in K$ satisfies $\zeta_3^2+\zeta_3+1=0$. Therefore $\mathcal{O}\otimes\ZZ_2\simeq \ZZ_2^4$ with quadratic form 
	\[(x,y,z,w)\mapsto Q(x,y,z,w)\coloneqq x^2+xy+y^2-2z^2-2zw-2w^2.\]
	We have 
	\begin{align*}
		Q(1,0,0,0)&\equiv Q(0,1,0,0)\equiv Q(1,1,1,0)\equiv Q(1,1,0,1)\equiv 1\pmod 2,\\
		Q(1,1,0,0)&\equiv Q(1,0,1,1)\equiv Q(0,1,1,2)\equiv Q(0,1,2,1)\equiv 3\pmod 8,\\
		Q(1,1,1,1)&\equiv Q(1,1,1,2)\equiv Q(1,1,2,1)\equiv Q(2,1,1,0)\equiv 5\pmod 8,\\
		Q(1,0,1,0)&\equiv Q(0,1,1,0)\equiv Q(1,0,0,1)\equiv Q(0,1,1,3)\equiv 7\pmod 8,
	\end{align*}
	and for each row, the four vectors define a matrix with odd determinant. So regardless of the value of~$s$, we can Hensel lift to obtain a basis for~$\ZZ_2^4$ consisting of elements of norm~$s$.
	
	\subsubsection{Eichler orders}\label{sec:eichler}
	An \emph{Eichler order} is an intersection of two maximal orders. If~$\mathcal{O}$ is an Eichler order in~$B_p$, then $\mathcal{O}\otimes\ZZ_p$ is maximal, and for $\ell\neq p$, $\mathcal{O}\otimes \ZZ_\ell$ is conjugate to $\begin{psmallmatrix}
		\ZZ_\ell & \ZZ_\ell \\ \ell^{r_\ell}\ZZ_\ell & \ZZ_\ell
	\end{psmallmatrix}$ for some $r_\ell\geq 0$~\cite[Section 23.4.19]{voight}. The exponent~$r_\ell$ is nonzero only for finitely many primes~$\ell$, and the product $\prod_\ell \ell^{r_\ell}$ is the \emph{index} of~$\mathcal{O}$ (equal to the index of~$\mathcal{O}$ in any maximal order containing it).
	
	Let~$\ell$ be a prime dividing the index of~$\mathcal{O}$ (we necessarily have $\ell\neq p$). If $\ell\mid s$, then the elements \[\begin{pmatrix}
		s&0\\0&1
	\end{pmatrix},\quad \begin{pmatrix}
		s&1\\0&1
	\end{pmatrix},\quad \begin{pmatrix}
		s&0\\\ell^{r_\ell}&1
	\end{pmatrix},\quad \begin{pmatrix}
		1&0\\ 0&s
	\end{pmatrix}\]
	each have norm~$s$ and form a basis for $\mathcal{O}\otimes \ZZ_\ell$. If $\ell\nmid 6s$ then the elements 
	\[\begin{pmatrix}
		s&0\\0&1
	\end{pmatrix},\quad \begin{pmatrix}
		s&1\\0&1
	\end{pmatrix},\quad \begin{pmatrix}
		s&0\\\ell^{r_\ell}&1
	\end{pmatrix},\quad \begin{pmatrix}
		2s&0\\ 0&\frac12
	\end{pmatrix}\]
	each have norm~$s$ and contain~$\begin{psmallmatrix}
		3s&0\\0&0
	\end{psmallmatrix}$ and~$\begin{psmallmatrix}
		0&0\\0&3
	\end{psmallmatrix}$ in their span, so they form a basis for $\mathcal{O}\otimes \ZZ_\ell$. So if the index of~$\mathcal{O}$ is not divisible by~$2$ or~$3$, then for all primes~$\ell$ there is a basis for $\mathcal{O}\otimes \ZZ_\ell$ consisting of elements of norm~$s$; by \cref{cor:basislocalglobal}, we can conclude that~$\mathcal{O}$ has a generating set with norms in~$S$.
	
	On the other hand, if the index of~$\mathcal{O}$ is even, then~$\mathcal{O}$ is not generated by elements of odd norm. This is because for $r\geq 1$, $\det\begin{psmallmatrix}
		x&y\\ 2^rz&w
	\end{psmallmatrix}\equiv 1\pmod 2$ implies $x\equiv w\equiv 1\pmod 2$, and the span of elements of this form lie in a proper sublattice of $\mathcal{O}\otimes\ZZ_2$. For a similar reason, if the index of~$\mathcal{O}$ is a multiple of~$3$ and $i=1$~or~$2$, then~$\mathcal{O}$ is not generated by elements of norm congruent to~$i\pmod 3$. So additional constraints on~$S$ are necessary if the index is not relatively prime to~$6$.

	\section{Theta function determines maximal order: background and setup}\label{sec:quat_background}

	We now turn to the second topic of this paper: determining maximal orders of~$B_p$ by their theta functions. In this section we discuss some examples of isospectral lattices and then set up some results about the lattice structure of quaternion orders; the main results \cref{thm:theta_to_iso} and \cref{thm:succmin_to_iso} are proven in \cref{sec:thetaproof}.

	\subsection{Isospectral Lattices}\label{sec:isospectral}
	
	Two integral lattices~$\Lambda,\Lambda'$ are \emph{isospectral} if their theta functions are equal. Schiemann proved that there do not exist any pairs of non-isometric isospectral lattices of rank at most~$3$ \cite{schiemann}, but there are many pairs of non-isometric but isospectral lattices of rank~$4$, including a four-parameter family due to Conway and Sloane \cite{conwaysloane92}.
	
	Now suppose we restrict to integral lattices in a quaternion algebra~$B_p$ whose left and right orders are maximal (every such lattice is locally similar to a maximal order of~$B_p$, and every quadratic form locally similar to a maximal order can be obtained in this way up to oriented similarity \cite[Section 19.6.7]{voight}). Equivalently, one can ask if a pair~$E,E'$ of supersingular elliptic curves over~$\overline{\FF_p}$ can be identified (up to Frobenius) by counting the number of isogenies of given degree from~$E$ to~$E'$. Even in this constrained setting it is common to find multiple lattices with the same theta function; the first instance of this occurring is for $p=67$. At $p=151$ we even find two non-isometric right ideals of the \emph{same} maximal order that have equal theta functions. These observations were explored in depth by Shiota~\cite{shiota}.
	
	\begin{ex}
		We include a brief description of the isospectral right ideals in the case $p=151$. Take $B_p=\QQ\langle 1,i,j,k\rangle$ with $i^2=-1$, $j^2=-151$, and $k\coloneqq ij=-ji$. Consider the maximal order
		\[O\coloneqq \left\langle   \frac12 + \frac12j + 4k, \frac1{32}i + \frac34j + \frac{69}{32}k, j + 8k, 16k \right\rangle.\]
		This order has right ideals 
		\begin{align*}
			I_1&\coloneqq \left\langle -5 + i - j - 3k, 10 - 42i + 2j - 2k, -7 + 11i + 5j - k, 74 + 22i + 2j - 2k\right\rangle,\\
			I_2&\coloneqq \left\langle -16 + 26i + 2k, 12 + 19i + 4j - k, 48 + 26i + 2k, -4 - 31i + 4j + 5k\right\rangle,
		\end{align*}
		each of norm~$512$, and we can check that~$I_1$ and~$I_2$ have non-isomorphic left orders. For each of~$I_1$ and~$I_2$, we take the basis $x_1,x_2,x_3,x_4$ given above and compute the corresponding (normalized) Gram matrix $\frac{1}{1024}(\Tr(x_i\overline{x_j}))_{1\leq i,j\leq 4}$, yielding
		\[\mathbf{A}_1=\frac12\begin{pmatrix}
			6&2&-1&1\\
			2&12&5&4\\
			-1&5&16&6\\
			1&4&6&28
		\end{pmatrix},\qquad \mathbf{A}_2=\frac12\begin{pmatrix}
			6&0&2&3\\
			0&12&3&4\\
			2&3&14&2\\
			3&4&2&28
		\end{pmatrix}.\]
		Both matrices have determinant~$\frac{151^2}{16}$, as expected (by \cite[Proposition 16.4.3]{voight} and \cref{eq:disc}). 
		The theta functions of both lattices begin with
		\[1 + 2q^3 + 2q^6 + 2q^7 + 2q^8 + 4q^9 + 2q^{10} + 2q^{11} + 4q^{12} + 4q^{28} + 4q^{15}+\cdots\]
		and are weight~$2$ cusp forms for~$\Gamma_0(151)$ \cite[Section 40.4.5]{voight}.  One can find a basis of $M_2(\Gamma_0(151))$ consisting of~$13$ elements, and check that under projection to the first~$13$ Fourier coefficients they remain independent. (The Sturm bound predicts that the first~$25$ coefficients are sufficient, but we can get by with fewer in this case.) Thus, if two elements of $M_2(\Gamma_0(151))$ agree on these coefficients, they must be equal. Hence, these lattices are isospectral. However, the lattices are not isometric because the vectors of norm~$3$ are orthogonal to the vectors of norm~$6$ in the second lattice but not in the first.
	\end{ex}

	\subsection{Lattice geometry of quaternion orders}\label{sec:lemmas}
	
	As before, let~$B_p$ denote the quaternion algebra over~$\QQ$ ramified at~$p$ and~$\infty$. There exists an isometry $B_p\otimes\RR\simeq \RR^4$, with the norm on~$B_p$ corresponding to the square of the standard Euclidean distance on~$\RR^4$, and $\frac12\Tr(x\bar y)$ giving the standard inner product $x\cdot y$.

	Let~$\Lambda$ be a lattice in~$B_p$ of rank $1\leq k\leq 4$; we say~$\Lambda$ is an \emph{integral lattice} if for all $x\in\Lambda$ we have $N(x),\Tr(x)\in\ZZ$. Given a basis $\{v_1,\ldots,v_k\}$ of an integral lattice~$\Lambda$ we can define a \emph{Gram matrix} 
	\[\mathbf{A}_\Lambda\coloneqq (v_i\cdot v_j)_{1\leq i,j\leq k}=(\tfrac12\Tr(v_i\overline{v_j}))_{1\leq i,j\leq k},\]
	and \emph{determinant} $\det \Lambda\coloneqq \det\mathbf{A}_\Lambda$ as in \cref{sec:notation}. The quaternion norm~$N$ on~$B_p$ defines an integral quadratic form on~$\Lambda$, and $\mathbf{A}_\Lambda\in \frac12M_k(\ZZ)$. 
	
	Now suppose~$\Lambda$ is an order in~$B_p$, so it is contained in a maximal order~$\mathcal{O}$ with finite index. We define the \emph{discriminant} of~$\Lambda$ to be $\disc \Lambda\coloneqq \det(2\mathbf{A}_{\Lambda})$. Note that the discriminant of~$\Lambda$ is a positive integer, while the determinant may not be an integer. In particular, if~$\mathcal{O}$ is a maximal order containing~$\Lambda$ then
	\begin{align}\label{eq:disc}
		\disc\Lambda=16\det\Lambda =[\mathcal{O}:\Lambda]^2p^2
	\end{align}
	\cite[Lemma 15.2.15 and Theorem 15.5.5]{voight}.
	If~$\Lambda$ has basis $v_1,\ldots,v_4$, then we have
	\begin{equation}\label{eq:tracemat}
		\disc \Lambda=\det(\Tr(v_i\bar{v_j}))_{1\leq i,j\leq 4}=|\det(\Tr(v_iv_j))_{1\leq i,j\leq 4}|
	\end{equation}
	\cite[Exercise 15.13]{voight}, where we call $(\Tr(v_iv_j))_{1\leq i,j\leq 4}$ the \emph{trace matrix} of the basis $\{v_1,\ldots,v_4\}$.

	\begin{defn}\label{def:succmin}
		Let~$\Lambda$ be a lattice of rank~$k$ with a positive definite quadratic form~$Q$. For $1\leq i\leq k$, the~$i$-th \emph{successive minimum} of~$\Lambda$ is the minimum value~$D_i$ such that the rank of the $\ZZ$-module generated by $\{v\in \Lambda:Q(v)\leq D_i\}$ is greater than or equal to~$i$. An ordered list $v_1,\ldots,v_k\in \Lambda$ \emph{attains the successive minima of~$\Lambda$} if it is linearly independent and $Q(v_i)=D_i$ for each $i=1,\ldots,k$.
	\end{defn}
	
	\begin{rmk}
		This is a non-standard definition, following the notation of \cite{chevyrev_galbraith}; if $(\Lambda,Q)$ is embedded isometrically in a Euclidean space~$\RR^n$, then the successive minima under \cref{def:succmin} are the squares of the corresponding successive minima under the standard definition. There always exists a list of elements attaining the successive minima, for instance by~\cite[Section VIII.1.2 Lemma 1]{cassels}.
	\end{rmk}

	
	\begin{lem}\label{lem:succminbasis}
		Let~$\Lambda$ be a lattice of rank $k\leq 3$. If a list of~$k$ vectors attains the sucessive minima of~$\Lambda$, then these vectors form a basis for~$\Lambda$. The same holds true for $k=4$ if we additionally assume that~$\Lambda$ is an order in~$B_p$ for~$p$ odd.
	\end{lem}
	Note that a counterexample for $p=2$ is given by the Hurwitz quaternions, $\mathcal{Z}\coloneqq \langle 1,i,j,\frac12(1+i+j+k)\rangle$ where $i^2=j^2=-1$ and $ij=k$. The elements $1,i,j,k$ attain the successive minima, but $\frac12(1+i+j+k)$ is not contained in their span.
	\begin{proof}
		Among all lattices of rank at most~$4$, the rank~$4$ cubic centered lattice~ $\mathbf{D}_4$ (which is isometric to~$\mathcal{Z}$ after rescaling)  is the only lattice up to similarity for which an arbitrary list of vectors attaining the successive minima is not always a basis~\cite[Corollary 6.2.3]{martinet}. If~$\Lambda$ is an order in~$B_p$ for~$p$ odd then the first successive minimum is equal to~$1$, so if~$\Lambda$ were a cubic centered lattice then its Gram matrix would have determinant~$\frac14$. This contradicts the fact that the determinant of an order in~$B_p$ must be an integer multiple of~$\frac{p^2}{16}$ by \cref{eq:disc}.
	\end{proof}
	
	\subsection{The Gross Lattice}\label{sec:gross_lattice}
	
	We define an additive map  $\tau\colon B_p\to B_p$ by
	\[\tau(x)=2x-\Tr(x).\] 
	If we restrict~$\tau$ to an order $\mathcal{O}\subseteq B_p$,  then the kernel of this map is~$\ZZ$, and the image is the \emph{Gross lattice} of~$\mathcal{O}$,
	\[\mathcal{O}^T\coloneqq \tau(\mathcal{O})=\{2x-\Tr(x):x\in\mathcal{O}\}\]
	(cf.~\cite[Section 12]{gross}).
	The Gross lattice is a strict subset of~$\mathcal{O}^0$, the subset of~$\mathcal{O}$ consisting of trace zero elements; more precisely, we have $\mathcal{O}^T=\mathcal{O}^0\cap(\ZZ+2\mathcal{O})$.
	Some relations between sublattices of~$\mathcal{O}$ are given below; inclusion arrows below are labeled by the index of one sublattice in the other, and~$\obot$ denotes orthogonal direct sum.
	\begin{equation}\label{eq:gross_containments}
		\begin{tikzcd} 		 
			\ZZ+ 2\mathcal{O}=\ZZ\obot\mathcal{O}^T\arrow[r,hook,"4"] &\ZZ\obot\mathcal{O}^0\arrow[r,hook,"2"]\arrow[d,"\tau"] & \mathcal{O}\arrow[d,"\tau"] & & \text{(rank 4)} \\
			& 2\mathcal{O}^0\arrow[r,hook,"2"] & \mathcal{O}^T \arrow[r,hook,"4"] & \mathcal{O}^0 & \text{(rank 3)}
		\end{tikzcd}
	\end{equation}
	The index of $\ZZ\obot\mathcal{O}^0$ in~$\mathcal{O}$ can be computed by noting that $\ZZ\obot\mathcal{O}^0$ consists of all elements of~$\mathcal{O}$ with even trace, and that~$\mathcal{O}$ must contain an element of odd trace (since $\det\mathcal{O}=\frac{p^2}{16}$ implies the Gram matrix of~$\mathcal{O}$ cannot be an integer matrix).

	\begin{rmk}
		Given a lattice $\Lambda\subseteq B_p$, define the \emph{dual lattice}~$\Lambda^\sharp$ by
		\begin{align*} 
			\Lambda^\sharp&\coloneqq\{y\in \QQ\Lambda:\Tr(x\bar{y})\in\ZZ\text{ for all }x\in\Lambda\}.
		\end{align*}
		While this will not be used in the rest of the paper, we note that~$\mathcal{O}^T$ can be related to~$(\mathcal{O}^\sharp)^0$, the trace zero part of the dual lattice of~$\mathcal{O}$. (This lattice arises in the correspondence between quaternion orders and ternary quadratic forms via the Clifford algebra construction; see \cite[Chapter 22]{voight}.) Specifically, we have the equality
		\[(\tfrac12\mathcal{O}^T)^\sharp=(\mathcal{O}^\sharp)^0,\]
		because when $\Tr(y)=0$ and $x\in\mathcal{O}$ we have $\Tr(\tau(x)\bar y)=2\Tr(x\bar y)$. Since~$\frac12\mathcal{O}^T$ is the orthogonal projection of~$\mathcal{O}$ onto~$B_p^0$, we can summarize this by saying that the dual of the projection equals the restriction of the dual.
	\end{rmk}

	The primary motivation for studying the Gross lattice is that elements of~$\mathcal{O}^T$ correspond to embeddings of imaginary quadratic orders in~$\mathcal{O}$. 
	Observe that for any $x\in B_p$ we have the equality
	\begin{equation}\label{eq:normtau}
		N(\tau(x))=4N(x)-\Tr(x)^2.
	\end{equation}
	So for any $\beta\in\mathcal{O}^T\setminus\{0\}$,~$-N(\beta)$ is equal to the discriminant of the quadratic order generated by a preimage of~$\beta$ under~$\tau$.

	Given an imaginary quadratic order~$R$ of discriminant~$-D$ for~$D>0$, an \emph{orientation} of~$R$ is a choice of $x\in R$ satisfying $x^2+D=0$ (which we will usually denote~$\sqrt{-D}$). Given two oriented orders $(R,x)$ and $(R,x')$, an \emph{oriented isomorphism} is an isomorphism $R\to R'$ sending $x\mapsto x'$. Note that for every imaginary quadratic discriminant~$-D$ there are exactly two oriented quadratic orders~$R$ up to oriented isomorphism, sent to each other by the nontrivial Galois action on~$R$.

	Given an imaginary quadratic order~$R$ and an embedding $\phi:R\hookrightarrow \mathcal{O}$, we say the embedding is \emph{optimal} if $\QQ\phi(R)\cap \mathcal{O}=\phi(R)$. Given a nonzero element~$v$ of a lattice~$L$, we say~$v$ is \emph{primitive} if there does not exist $w\in L$ and $n\geq 2$ with $v=nw$.  The following proposition is implicit in \cite[Proposition 12.9]{gross}; we include the proof for completeness. \black
	
	\begin{prop}\label{lem:embeddingsnorms}
		There is a one-to-one correspondence between nonzero elements $\beta\in\mathcal{O}^T$ and embeddings $R\hookrightarrow\mathcal{O}$  of oriented imaginary quadratic orders~$R$ up to oriented isomorphism.  Under this correspondence we have $\disc R=-N(\beta)$, and the embedding $R\hookrightarrow \mathcal{O}$ is optimal if and only if the corresponding~$\beta$ is primitive.
	\end{prop}
	\begin{proof}
		The imaginary quadratic order $R=\ZZ[\frac12(D+\sqrt{-D})]$ has discriminant $-D$ and orientation determined by the element~$\sqrt{-D}$. To an embedding $\phi\colon R\to\mathcal{O}$ we associate the element 
		\[\beta\coloneqq \phi(\sqrt{-D})=2\phi\left(\frac{D+\sqrt{-D}}{2}\right)-\Tr\left(\phi\left(\frac{D+\sqrt{-D}}{2}\right)\right)\in\mathcal{O}^T.\]
		\black
		Conversely, given an element $\beta=2x-\Tr(x)$ of~$\mathcal{O}^T$, set $D=N(\beta)$ and associate to~$\beta$ the embedding of $\ZZ[\frac12 (D+\sqrt{-D})]$ into~$\mathcal{O}$ determined by 
		\begin{align*}
			\frac{D+\sqrt{-D}}{2}&\mapsto x+\frac{D-\Tr(x)}{2}.
		\end{align*}
		We have $x+\frac{D-\Tr(x)}{2}\in\mathcal{O}$ because $D\equiv \Tr(x)\pmod 2$ by \cref{eq:normtau}. Further, both sides have trace~$D$, and applying \cref{eq:normtau} to the fact that $\tau\left(x+\frac{D-\Tr(x)}{2}\right)=\beta$ we find that both sides have norm $\frac14 (D^2+D)$. \black		
		Hence, this map is well-defined.
		It is straightforward to verify that these two associations are inverses and that the remaining claims are satisfied.
	\end{proof}
	
	Note that if two embeddings $R\to\mathcal{O}$ are related by Galois conjugation on~$R$, then the corresponding elements of~$\mathcal{O}^T$ are negatives of each other. 

	For any $\beta\in \mathcal{O}^T$, there exists a unique $\alpha\in\mathcal{O}$ satisfying $\tau(\alpha)=\beta$ and $\Tr(\alpha)\in\{0,1\}$ (explicitly, we can let $\delta\in\{0,1\}$ satisfy $\delta\equiv N(\beta)\pmod 2$ and set $\alpha=\frac12(\delta+\beta)$). By \cref{eq:normtau}, this element~$\alpha$ attains the minimal norm among all elements of~$\mathcal{O}$ mapping under~$\tau$ to~$\beta$.
	
	\begin{lem}\label{lem:OTsucmin_to_Osucmin}
		Let $\beta_1,\beta_2,\beta_3\in\mathcal{O}^T$ be linearly independent, and let $\alpha_1,\alpha_2,\alpha_3\in\mathcal{O}$ satisfy $\tau(\alpha_i)=\beta_i$ and $\Tr(\alpha_i)\in\{0,1\}$ for each~$i$. The following are equivalent:
		\begin{enumerate}[label=(\alph*)]
			\item The successive minima for~$\mathcal{O}^T$ are attained by $\beta_1,\beta_2,\beta_3$.
			\item The successive minima for~$\mathcal{O}$ are attained by $1,\alpha_1,\alpha_2,\alpha_3$, and if $\Tr(\alpha_i)=0$ for some $i=1,2,3$, then for all $\gamma\in\mathcal{O}$ with $N(\gamma)=N(\alpha_i)$ that are linearly independent from $1,\alpha_1,\ldots,\alpha_{i-1}$, we have $\Tr(\gamma)=0$. 
		\end{enumerate}
	\end{lem}
	\begin{proof}
		Observe that $1,\alpha_1,\alpha_2,\alpha_3$ are linearly independent, because applying~$\tau$ to a linear dependence would induce a dependence among $\beta_1,\beta_2,\beta_3$. Now assume (a). For $i=1,2,3$, let~$S_i$ denote the set of $\gamma\in\mathcal{O}$ satisfying $N(\gamma)<N(\alpha_i)$. For any $\gamma\in S_i$ we have
		\[N(\tau(\gamma))=4N(\gamma)-\Tr(\gamma)^2\leq 4(N(\alpha_i)-1)<4N(\alpha_i)-\Tr(\alpha_i)^2=N(\beta_i),\]
		and since~$N(\beta_i)$ is the~$i$-th successive minimum for~$\mathcal{O}^T$, the span of~$\tau(\gamma)$ for all $\gamma\in S_i$ has dimension at most~$i-1$. Thus the span of~$S_i$ has dimension at most~$i$, proving by induction on~$i$ that $1,\alpha_1,\alpha_2,\alpha_3$ attain the successive minima for~$\mathcal{O}$. Now for the sake of contradiction suppose that $\Tr(\alpha_i)=0$, and there exists $\gamma\in\mathcal{O}$ linearly independent from $1,\alpha_1,\ldots,\alpha_{i-1}$ with $N(\gamma)=N(\alpha_i)$ and  $\Tr(\gamma)\neq 0$. \black Then $\beta_1,\ldots,\beta_{i-1},\tau(\gamma)$ are~$i$ independent elements of~$\mathcal{O}^T$, but
		\[N(\tau(\gamma))=4N(\gamma)- \Tr(\gamma)^2 <4N(\alpha_i)=N(\beta_i),\]
		contradicting the assumption that~$\beta_i$ attains the~$i$-th successive minimum for~$\mathcal{O}^T$. Thus (a) implies (b).
		
		Now assume (b). For $i=1,2,3$, let~$S_i$ denote the set of $x\in\mathcal{O}^T$ satisfying $N(x)<N(\beta_i)$. For any $x\in S_i$ there exists $\gamma\in\mathcal{O}$ with $\tau(\gamma)=x$ and $\Tr(\gamma)\in\{0,1\}$. If $N(x)<N(\beta_i)-1$, then in fact $N(x)\leq N(\beta_i)-3$ because the norm of every element of~$\mathcal{O}^T$ is either~$0$~or~$3$~mod~$4$  by \cref{eq:normtau},  and so
		\[N(\gamma)=\frac14(N(\tau(\gamma))+\Tr(\gamma)^2)\leq\frac14(N(\beta_i)-2)< N(\alpha_i).\]
		Since~$N(\alpha_i)$ is the $(i+1)$-st successive minimum for~$\mathcal{O}$, this implies $1,\alpha_1,\ldots,\alpha_{i-1},\gamma$ must be linearly depenendent. On the other hand, suppose $N(x)=N(\beta_i)-1$. Then $N(\gamma)=N(\alpha_i)$, $\Tr(\gamma)=1$, and $\Tr(\alpha_i)=0$, so once again $1,\alpha_1,\ldots,\alpha_{i-1},\gamma$ are linearly dependent. Either way we can conclude that $x=\tau(\gamma)$ is in the~$\QQ$-span of $\beta_1,\ldots,\beta_{i-1}$. This shows that the span of~$S_i$ has dimension less than~$i$, so~$N(\beta_i)$ is indeed the~$i$-th successive minimum for~$\mathcal{O}^T$.
	\end{proof}
	
	\begin{rmk}
		It is possible for $1,\alpha_1,\alpha_2,\alpha_3$ to attain the successive minima of $\mathcal{O}$, but $\tau(\alpha_1),\tau(\alpha_2),\tau(\alpha_3)$ not attain the successive minima of $\mathcal{O}^T$. A simple example is given by the Hurwitz quaternions $\mathcal{Z}= \langle 1,i,j,\frac12(1+i+j+k)\rangle$ with $i^2=j^2=-1$ and $ij=k$.  The successive minima of $\mathcal{Z}$ are all $1$, and the successive minima of $\mathcal{Z}^T$ are all $3$. The elements $1,i,j,k$ attain the successive minima of $\mathcal{O}$, but $N(\tau(i))=N(\tau(j))=N(\tau(k))=4$, and so $\tau(i),\tau(j),\tau(k)$ do not realize the successive minima of $\mathcal{Z}^T$. And indeed, condition (b) of \cref{lem:OTsucmin_to_Osucmin} does not hold; for any of $i=1,2,3$, we can take $\gamma=\frac12(1+i+j+k)$. \black
	\end{rmk}

	\black	
	
	\begin{lem}\label{lem:smallcases}
		Let $D_1<15$. Up to isomorphism, there is at most one maximal order $\mathcal{O}\subseteq B_p$ such that the first successive minimum of~$\mathcal{O}^T$ is equal to~$D_1$.
	\end{lem}
	\begin{proof}
		If~$D_1$ is not~$0$~or~$3$~mod~$4$ then there is no maximal order with first successive minimum~$D_1$ by \cref{lem:embeddingsnorms}. For all remaining~$D_1$, the quadratic order of discriminant~$-D_1$ has class number~$1$; in this case there is a unique maximal order (up to isomorphism) in which this quadratic order embeds~\cite[Corollary 30.4.23]{voight}, and therefore a unique~$\mathcal{O}$ (up to isomorphism) such that~$\mathcal{O}^T$ has an element of norm~$D_1$. 
	\end{proof}

	\begin{rmk}
		If~$-D$ is a fundamental discriminant, an explicit maximal order admitting an optimal embedding of the ring of integers of discriminant~$-D$ can be written down explicitly using \cite[Equation (5)]{dorman} (see also \cite[Theorem 1]{dorman}).
	\end{rmk}

	\subsection{Constraints on short Gross lattice vectors}\label{sec:smalldisc_constraints}
	
	A key idea we will apply is that there are very strict constraints on  arrangements of short elements in the Gross lattice.  This idea can be made precise using a construction due to Kaneko~\cite{kaneko}, which we present as \cref{prop:kaneko} with minor modifications for our convenience. Kaneko used this construction to prove a bound on the discriminants of quadratic orders embedding into a quaternion order, a special case of which is given by \cref{lem:distinct_embeddings}. In addition to this bound (a constraint on  the norms of independent elements in~$\mathcal{O}^T$ ), we also establish a constraint on  the angle between two elements of~$\mathcal{O}^T$,  \cref{lem:uniquetr}.\footnote{This is the only part of the argument that relies the fact that our quaternion algebra~$B_p$ is ramified at a single prime; if we consider orders in definite quaternion algebras ramified at multiple primes, the corresponding constraint becomes much less strict.} Together, these constraints will be sufficient to show that a quaternion order is uniquely determined up to isomorphism by the  numbers of short vectors in~$\mathcal{O}^T$. 
	
	\begin{prop}\label{prop:kaneko}
		Let~$\mathcal{O}$ be an order in~$B_p$, and let~$\beta_1,\beta_2\in\mathcal{O}^T$ be linearly independent. Then
		\[
		N(\beta_1)N(\beta_2)-\frac14\Tr(\beta_1\bar{\beta_2})^2
		\]
		is a positive integer multiple of~$4p$. 
	\end{prop}
	The number $N(\beta_1)N(\beta_2)-\frac14\Tr(\beta_1\bar{\beta_2})^2$ is the determinant of the Gram matrix of the lattice $\langle \beta_1,\beta_2\rangle$, so this can be interpreted as saying that every parallelogram in~$\mathcal{O}^T$ has area $2\sqrt{kp}$ for some positive integer~$k$. 
	A version of this is proven by Kaneko in \cite[Section 3]{kaneko}; see also equation (3.2) of \cite{chevyrev_galbraith}.  A similar idea also appears in \cite{goren_lauter}.
	\begin{proof}
		The value is positive because it is the determinant of the Gram matrix of the lattice $\langle \beta_1,\beta_2\rangle$, so it suffices to show divisibility by~$4p$. For $i=1,2$, set $D_i=N(\beta_i)$, and let $\alpha_i\in\mathcal{O}$ be a minimal norm preimage of~$\beta_i$ under~$\tau$; that is, take $\delta_i\in\{0,1\}$ with $\delta_i\equiv D_i\pmod 2$ and set $\alpha_i=\frac12(\beta_i+\delta_i)$. 
		Define an order 
		\[\Lambda=\langle 1,\alpha_1,\alpha_2,\alpha_1\alpha_2\rangle\subseteq \mathcal{O}.\] Letting $s=\Tr(\alpha_1\alpha_2)$, one can compute the trace matrix of~$\Lambda$ (\cref{eq:tracemat}),
		\[\mathbf{T}=\begin{pmatrix}
			2 & \delta_1 & \delta_2 & s\\
			\delta_1 & \delta_1^2-2N(\alpha_1) & s & -\delta_2N(\alpha_1)+\delta_1s\\
			\delta_2 & s & \delta_2^2-2N(\alpha_2) &  -\delta_1N(\alpha_2)+\delta_2s\\
			s & -\delta_2N(\alpha_1)+\delta_1s & -\delta_1N(\alpha_2)+\delta_2s & s^2 - 2N(\alpha_1)N(\alpha_2)
		\end{pmatrix}.\]
		Since $N(\alpha_i)=\frac14(\delta_i+D_i)$, we compute
		\[\disc\Lambda=|\det\mathbf{T}|=\frac{(D_1D_2-(2s-\delta_1\delta_2)^2)^2}{16}.\]
		Noting that
		\begin{align}\label{eq:eventrace}
			\Tr(\beta_1\beta_2)=\Tr((2\alpha_1-\delta_1)(2\alpha_2-\delta_2))=4\Tr(\alpha_1\alpha_2)-2\delta_1\delta_2,
		\end{align}
		we can replace $2s-\delta_1\delta_2$ with $\frac12\Tr(\beta_1\beta_2)=-\frac12\Tr(\beta_1\bar{\beta_2})$. Since~$\Lambda$ is an order in~$B_p$, we can use \cref{eq:disc} to conclude that $\frac14(D_1D_2-\frac14\Tr(\beta_1\bar{\beta_2})^2)$ is an integer multiple of~$p$.
	\end{proof}
	
	Some immediate consequences of this calculation are as follows.
	
	\begin{cor}\label{lem:succminbounds}
		Let~$\mathcal{O}$ be an order in~$B_p$ of discriminant $\Delta\in\ZZ$.
		Let $D_1,D_2,D_3$ be the successive minima of~$\mathcal{O}^T$. Then  
		\begin{align*}
			D_1&\leq 2\Delta^{1/3},\\
			\frac{4p}{D_1}\leq D_2&\leq\left(\frac{8\Delta}{D_1}\right)^{1/2}.
		\end{align*}
	\end{cor}
	\begin{proof}
		Since $\det\mathcal{O}=\frac\Delta{16}$, \cref{eq:gross_containments} implies that 
		\[\det\mathcal{O}^T=\det(\ZZ\obot\mathcal{O}^T)= 8^2\det\mathcal{O}= 4\Delta.\]
		(Recall that $\det\Lambda$ refers to the determinant of the Gram matrix of~$\Lambda$, so $\Lambda'\subseteq\Lambda$ implies $\det\Lambda'=[\Lambda:\Lambda']^2\det\Lambda$.) 
		We also have $D_1D_2D_3\leq 2\det \mathcal{O}^T$ (a bound specific to rank~$3$ lattices) by \cite[Lecture XI (25)]{siegel}.		
		The desired upper bounds follow from $D_1^3\leq D_1D_2D_3$ and $D_1D_2^2\leq D_1D_2D_3$. The lower bound $D_1D_2\geq 4p$ follows from \cref{prop:kaneko}.
	\end{proof}
	
	\begin{cor}\label{lem:distinct_embeddings}
		If a quadratic order~$R$ has two embeddings in~$\mathcal{O}$ with distinct images, then $\disc R> p$.
	\end{cor}
	\begin{proof}
		This is a special case of \cite[Theorem 2']{kaneko}. Let $D\coloneqq \disc R$, and $\beta_1,\beta_2\in\mathcal{O}^T$ be the elements corresponding to the two embeddings of~$R$ under \cref{lem:embeddingsnorms}. Then from \cref{prop:kaneko} we obtain 
		\[p\mid \left(\frac{D+t}{2}\right)\left(\frac{D-t}{2}\right)\]
		where $t=\frac12\Tr(\beta_1\bar\beta_2)\in\ZZ$ by \cref{eq:eventrace}. Each factor is an integer: $D+t$ and $D-t$ have the same parity, and since their product is a multiple of~$4$, both must be even. Thus~$p$ divides one of the factors, so $p\leq \frac12(D+t)$. Since $D^2-t^2>0$ we have $t<D$, so $p<D$.
	\end{proof}
	
	
	Recall that $(v_1,v_2)\mapsto \frac12\Tr(v_1\overline{v_2})$ defines an inner product on~$B_p$. The following result says that if two elements of~$\mathcal{O}^T$ are sufficiently small, then their norms uniquely determine the angle between them (up to negating either element). This is one of the most important conceptual ingredients of the proofs of \cref{thm:theta_to_iso} and \cref{thm:succmin_to_iso}: while the theta function records lengths of elements but loses all information about angles between them, this result allows us to recover information about angles from information about lengths.
	
	\begin{cor}\label{lem:uniquetr}
		Let $\beta_1,\beta_2$ be independent elements of~$\mathcal{O}^T$. Suppose $N(\beta_1)\leq p$, and that ~$\beta_2$ has minimal norm in $\beta_2+\ZZ\beta_1$.  Then $|\frac12\Tr(\beta_1\bar{\beta_2})|$ equals the smallest positive square root of $N(\beta_1)N(\beta_2)$ modulo~$p$.
	\end{cor}
	\begin{proof}
		By \cref{prop:kaneko} we have $\frac14\Tr(\beta_1\bar{\beta_2})^2\equiv N(\beta_1)N(\beta_2)\pmod{4p}$,  so that $|\frac12\Tr(\beta_1\bar{\beta_2})|$, an integer by \cref{eq:eventrace}, is a square root of $N(\beta_1)N(\beta_2)$ modulo~$p$.		
		Expanding $N(\beta_2\pm \beta_1)-N(\beta_2)\geq 0$, we obtain $\mp \Tr(\beta_1\bar{\beta_2})\leq N(\beta_1)$; since this is true for both choices of sign, we have 
		\[0\leq|\tfrac12\Tr(\beta_1\bar \beta_2)|\leq \frac{N(\beta_1)}{2}\leq \frac p2.\]
		There is a unique square root of~$D_1D_2$~modulo~$p$ in this interval.
	\end{proof}
	
	\section{Theta function determines maximal order}\label{sec:thetaproof}

	In this section we prove \cref{thm:succmin_to_iso} and \cref{thm:theta_to_iso}.	We begin in \cref{sec:mins_to_order} with a proof of \cref{thm:succmin_to_iso}, the statement that the successive minima of~$\mathcal{O}^T$ determine the isomorphism type of~$\mathcal{O}$. So to prove \cref{thm:theta_to_iso}, all that remains to show is that the theta function of~$\mathcal{O}$ determines the successive minima $D_1,D_2,D_3$ of~$\mathcal{O}^T$. In \cref{sec:theta_to_mins} we introduce a decomposition of the theta function of~$\mathcal{O}$. Using this decomposition and the results of \cref{sec:smalldisc_constraints}, we show that the theta function of~$\mathcal{O}$ determines~$D_1$ and~$D_2$, and in \cref{sec:determining_D3} we show that the theta function of~$\mathcal{O}$ determines~$D_3$.

	\subsection{Successive minima of Gross lattice determines the order}\label{sec:mins_to_order}
	
	Let~$p$ be an odd prime, and~$\mathcal{O}$ an order in~$B_p$ of discriminant~$r^2p^2$. Let $\beta_1,\beta_2,\beta_3\in\mathcal{O}^T$ attain the successive minima $D_1\leq D_2\leq D_3$ of~$\mathcal{O}^T$, and assume  $D_1\geq 8r^2$.  We will begin by showing that~$\mathcal{O}^T$ is determined up to isometry by $D_1,D_2,D_3$. 
	
	For each pair $1\leq i<j\leq 3$, let $0\leq T_{ij}\leq \frac{p}{2}$ be the unique integer satisfying $T_{ij}^2\equiv D_iD_j\pmod p$.  Using \cref{lem:succminbounds} we have 
	\[N(\beta_1)\leq N(\beta_2)\leq \sqrt{\frac{8r^2p^2}{D_1}}\leq p,\]
	so by \cref{lem:uniquetr}, we have $|\frac12\Tr(\beta_i\bar{\beta_j})|=T_{ij}$ for each pair $i,j$.
	
	Since $\beta_1,\beta_2,\beta_3$ attain successive minima for the rank~$3$ lattice~$\mathcal{O}^T$, they form a basis by \cref{lem:succminbasis}. Let $\mathbf{A}=(\frac12\Tr(\beta_i\bar{\beta_j}))_{i,j}$ be the corresponding Gram matrix. By replacing~$\beta_i$ with $-\beta_i=\bar{\beta_i}$ if necessary for some values of~$i$, we can ensure that any two of the equations 
	\[\frac12\Tr(\beta_1\bar{\beta_2})=T_{12},\qquad\frac12\Tr(\beta_1\bar{\beta_3})=T_{13},\qquad\frac12\Tr(\beta_2\bar{\beta_3})=T_{23}\] 
	hold. So if in addition we have $T_{ij}=0$ for some pair $i,j$, then we can choose $\beta_1,\beta_2,\beta_3$ so that $\frac12\Tr(\beta_i\bar{\beta_j})=T_{ij}$ for all $i,j$.  Thus,~$\mathcal{O}^T$ is determined up to isometry.
	
	On the other hand, suppose all~$T_{ij}$ are nonzero. Then without loss of generality we have either $\mathbf{A}=\mathbf{A}_+$ or $\mathbf{A}=\mathbf{A}_-$, where
	\[\mathbf{A}_\pm\coloneqq \begin{pmatrix}
		N(\beta_1)& T_{12}&\pm  T_{13}\\
		T_{12}&N(\beta_2)& T_{23}\\
		\pm  T_{13}& T_{23}&N(\beta_3)
	\end{pmatrix}.\]
	Now the orthogonal direct sum $\ZZ\obot \mathcal{O}^T$ is a sublattice of~$\mathcal{O}$, and so $16\det \mathbf{A}$ must be a multiple of~$p^2$. But we have
	\[16\det(\mathbf{A}_+)-16\det(\mathbf{A}_-)=4T_{12}T_{23}T_{13}.\]
	Since the integers~$T_{ij}$ satisfy $0<T_{ij}\leq\frac{p}{2}$, this difference is not a multiple of~$p$, and therefore only one of $16\det(\mathbf{A}_+)$ and $16\det(\mathbf{A}_-)$ can be a multiple of~$p$. This determines~$\mathbf{A}$ uniquely, and so again,~$\mathcal{O}^T$ is determined up to isometry.  If we replace $\beta_1,\beta_2,\beta_3$ with $-\beta_1,-\beta_2,-\beta_3$, then we preserve~$\mathcal{O}^T$ (and~$\mathbf{A}$) but obtain a basis with opposite orientation.  Thus~$\mathcal{O}^T$ is determined up to orientation-preserving isometry.
	
	Now let $\mathcal{O}_1,\mathcal{O}_2$ be two orders in~$B_p$, each of index~$r$ in some (perhaps different) maximal order. Suppose that~$\mathcal{O}_1^T$ and~$\mathcal{O}_2^T$ have the same successive minima. We established above that there exists an orientation-preserving isometry $\varphi:\mathcal{O}_1^T\to\mathcal{O}_2^T$, which extends by linearity to an orientation-preserving isometry $\varphi:B_p^0\to B_p^0$ on the trace~$0$ subspace of~$B_p$. Every such isometry can be written as a conjugation map $\varphi(x)=\gamma^{-1}x\gamma$ for some $\gamma\in B_p^\times$ \cite[Proposition 4.5.10]{voight}. Finally, by a result of Chevyrev and Galbraith~\cite[Lemma 4]{chevyrev_galbraith}, the conjugation map $\varphi:\mathcal{O}_1^T\to\mathcal{O}_2^T$ extends to an isomorphism $\mathcal{O}_1\to\mathcal{O}_2$. This proves \cref{thm:succmin_to_iso}. 
	
	\subsection{Decomposing the theta series along fibers of~$\tau$}\label{sec:theta_to_mins}
	
	
	Let  $\tau\colon B_p\to B_p^0$ denote the map $\tau(x)=2x-\Tr(x)$. Given any integral lattice $L\subseteq B_p$ that contains~$\ZZ$, the fibers of~$\tau$ partition~$L$ into cosets of~$\ZZ$, allowing us to decompose the theta function of~$L$ into a sum over these cosets.  
	
	Define power series $\theta_0,\theta_1\in \ZZ[[q]]$ by
	\begin{align*}
		\theta_0(q)&=\sum_{n\in\ZZ} q^{n^2}=1+2q+2q^4+2q^9+\cdots,\\
		\theta_1(q)&=\sum_{n\in\ZZ} q^{n^2+n}=2+2q^2+2q^6+2q^{12}+\cdots.
	\end{align*}
	
	\begin{lem}\label{lem:thetadecomp}
		Given an integral lattice $L\supseteq\ZZ$, we have
		\[\theta_{L}(q)=\left(\sum_{\substack{\beta\in \tau(L)\\N(\beta)\equiv 0\;\mathrm{mod}\;4}} q^{N(\beta)/4}\right)\theta_0(q)+\left(\sum_{\substack{\beta\in \tau(L)\\N(\beta)\equiv 3\;\mathrm{mod}\;4}} q^{(1+N(\beta))/4}\right)\theta_1(q).\]
	\end{lem}
	In the case $L=\ZZ$ we have $\tau(L)=\{0\}$, and this reduces to the trivial observation $\theta_\ZZ(q)=\theta_0(q)$.
	\begin{proof}
		We can write
		\[\theta_{L}(q)=\sum_{x\in L}q^{N(x)}=\sum_{\beta\in\tau(L)}\sum_{x\in \tau^{-1}(\beta)}q^{N(x)}.\]	
		For all $\beta\in\tau(L)$,~$N(\beta)$ is either~$0$~or~$3$~mod~$4$ by \cref{eq:normtau}. We will determine the sum of~$q^{N(x)}$ over~$x$ in $\tau^{-1}(\beta)$ in each of these two cases.
		
		If $N(\beta)\equiv 0\pmod 4$, then every $x\in\tau^{-1}(\beta)$ has even trace (cf.~\cref{lem:embeddingsnorms}), so $\tau^{-1}(\beta)=\{\frac12\beta+n:n\in\ZZ\}$. Since~$\beta$ is orthogonal to~$1$ we have
		\[\sum_{x\in \tau^{-1}(\beta)}q^{N(x)}=\sum_{n\in\ZZ}q^{N(\beta/2)+n^2}=q^{N(\beta)/4}\theta_0(q).\]
		
		If $N(\beta)\equiv 3\pmod 4$, then every $x\in\tau^{-1}(\beta)$ has odd trace, so that $\tau^{-1}(\beta)=\{\frac12\beta+n+\frac12:n\in\ZZ\}$. Then
		\[
		\sum_{x\in \tau^{-1}(\beta)}q^{N(x)}=\sum_{n\in\ZZ}q^{N(\beta/2)+(n+\frac12)^2}=q^{(N(\beta)+1)/4}\theta_1(q).\qedhere\]
	\end{proof}

	\begin{rmk}
		By identifying each of the terms in \cref{lem:thetadecomp} as the even or odd parts of appropriate theta functions, we can rewrite the equality more elegantly as
		\begin{align*}
			\theta_L(q^4)&={\color{white}+}\tfrac14\big(\theta_{\tau(L)}(q)+\theta_{\tau(L)}(-q)\big)\big(\theta_{\ZZ}(q)+\theta_{\ZZ}(-q)\big)\\
			&{\color{white}=} +\tfrac14\big(\theta_{\tau(L)}(q)-\theta_{\tau(L)}(-q)\big)\big(\theta_{\ZZ}(q)-\theta_{\ZZ}(-q)\big)\\
			&={\color{white}+}\tfrac12\big(\theta_{\tau(L)}(q)\theta_{\ZZ}(q)+\theta_{\tau(L)}(-q)\theta_{\ZZ}(-q)\big);
		\end{align*}
		This can also be obtained by recognizing~$2L$ as the set of elements of even norm in $\ZZ\obot\tau(L)$ (as in \cref{eq:gross_containments}). However, the form given in the lemma statement will be more convenient, as it displays more clearly the contributions to $\theta_L(q)$ from individual elements of $\tau(L)$.
	\end{rmk}

	One consequence of this lemma is that the theta series of  an integral lattice containing~$\ZZ$ can be determined from the theta series of its image under~$\tau$.  The difficulty is recovering information about~$\tau(L)$ from the theta series of~$L$: the power series $f(q),g(q)\in\ZZ[[q]]$ such that $\theta_{L}(q)=f(q)\theta_0(q)+g(q)\theta_1(q)$ are far from being unique.

	\begin{ex}\label{ex:tau_nonisom}
		Let~$B_3$ be the quaternion algebra over~$\QQ$ ramified at~$3$, defined by $i^2=-1$ and $j^2=k^2=-3$ (with $k=ij$). Consider the two integral lattices
		\begin{align*}
			L_1\coloneqq \left\langle 1,i,\frac{1+j}{2},k\right\rangle,\qquad
			L_2\coloneqq \left\langle 1,i,\frac{i+j}{2},k\right\rangle.
		\end{align*}
		These lattices are isometric (via swapping~$1$ with~$i$) and therefore have the same theta function. However, their images under~$\tau$,
		\begin{align*}
			\tau(L_1)= \left\langle 2i,j,2k\right\rangle,\qquad
			\tau(L_2)= \left\langle 2i,i+j,2k\right\rangle,
		\end{align*}
		are not isometric; they have different successive minima ($3,4,12$ and $4,4,12$ respectively) and different theta functions. This demonstrates that the decomposition of a theta series as in \cref{lem:thetadecomp} is not unique, and that we can not in general determine the lattice structure of~$\tau(L)$ from the lattice structure of~$L$ alone.
	\end{ex}

	From now on, we suppose~$\mathcal{O}$ is a maximal order. Our goal in the remainder of the paper is to use the geometry of~$\mathcal{O}$  to obtain constraints on the terms appearing in \cref{lem:thetadecomp}, and so deduce the successive minima of~$\mathcal{O}^T$.  The strategy is to start with $L=\ZZ$ and inductively build up a lattice $\ZZ\subseteq L\subseteq\mathcal{O}$ with known structure, one dimension at a time. If~$c_nq^n$ is the smallest nonzero term of $\theta_\mathcal{O}(q)-\theta_L(q)$, we can conclude that there are~$c_n$ elements of norm~$n$ in $\mathcal{O}\setminus L$, and no shorter elements. We then use general properties of quaternion orders to show that the traces of these elements can be determined; this allows us to determine the minimal polynomial of an element $\alpha\in\mathcal{O}$ of norm~$n$ whose image under~$\tau$ attains the next successive minimum of~$\mathcal{O}^T$. Finally, we can use \cref{lem:uniquetr} to determine the full lattice structure of $L+\langle\alpha\rangle$.
	
	\subsection{Determining~$D_1$ and~$D_2$}\label{sec:D1D2}
	
	As above, let $\mathcal{O}\subseteq B_p$ be a maximal order, and let $D_1,D_2,D_3$ denote the successive minima of~$\mathcal{O}^T$. If $p=2,3,5,7$, then there is a unique maximal order in~$B_p$ (for instance by \cite[Exercise 30.6]{voight}), so the isomorphism type of~$\mathcal{O}$ (and in particular the successive minima of~$\mathcal{O}^T$) are uniquely determined. Thus from now on we can assume $p\geq 11$.

	
	\begin{lem}\label{lem:determine_D1}
		Let~$c_nq^n$ denote the first nonzero term of $\theta_\mathcal{O}(q)-\theta_\ZZ(q)$. Then one of the following occurs:
		\begin{itemize}
			\item $c_n=2$, in which case $D_1=4n$ and  $D_2,D_3\geq 4n+3$.
			\item $c_n=4$, in which case $D_1=4n-1$ and $D_2,D_3\geq 4n+3$.
			\item $c_n=6$, in which case $D_1=4n-1$, $D_2=4n$, and $D_3\geq 4n+3$.
		\end{itemize}
	\end{lem}
	\begin{proof}
		By \cref{lem:thetadecomp}, the term~$c_nq^n$ has contributions from $\beta\in\mathcal{O}^T\setminus\{0\}$ with norm~$4n$ or~$4n-1$.  Since~$n$ is minimal these elements are primitive, so  by \cref{lem:embeddingsnorms} they correspond to optimal embeddings of quadratic orders of discriminant~$4n$ and~$4n-1$, respectively. We have $4n-1\leq D_1\leq  2p^{2/3}$ by \cref{lem:succminbounds}, which implies $4n\leq p$ since  $p\geq 11$.  So by \cref{lem:distinct_embeddings}, there cannot exist $\alpha,\alpha'\in\mathcal{O}$ both of norm~$n$ but generating distinct isomorphic subfields of~$\mathcal{O}$. Hence there are only three options: only $\ZZ[\sqrt{-n}]$ optimally embeds in~$\mathcal{O}$, only $\ZZ[\frac{1+\sqrt{1-4n}}{2}]$ optimally embeds, or both optimally embed. These three cases can each be identified by counting the number of norm~$n$ elements in each quadratic order.
	\end{proof}
	
	If $D_1<15$, then the isomorphism type of~$\mathcal{O}$ is uniquely determined by \cref{lem:smallcases}; in particular the remaining successive minima~$D_2$ and~$D_3$ of~$\mathcal{O}^T$ are also determined. So from now on we assume $D_1\geq 15$.
	Using \cref{lem:determine_D1}, we use~$\theta_\mathcal{O}$ to deduce the existence of an element $\alpha_1\in\mathcal{O}$ with norm~$n$ and trace either~$0$ or~$1$, depending on the parity of~$D_1$; hence we can determine the structure of $\ZZ[\alpha_1]$.
	
	\begin{lem}\label{lem:determine_D2}
		Suppose $D_1\geq 15$, and let $c_nq^n$ denote the first nonzero term of $\theta_{\mathcal{O}}(q)-\theta_{\ZZ[\alpha_1]}(q)$. Then one of the following occurs:
		\begin{itemize}
			\item $c_n=2$, in which case $D_2=4n$ and $D_3\geq 4n+3$.
			\item $c_n=4$, in which case $D_2=4n-1$ and $D_3\geq 4n+3$.
			\item $c_n=6$, in which case $D_2=4n-1$ and $D_3=4n$.
		\end{itemize}
	\end{lem}
	\begin{proof}
		The term~$c_nq^n$ has contributions from $\beta\in\mathcal{O}^T\setminus \tau(\ZZ[\alpha_1])$ with norm~$4n$ or~$4n-1$.  We can use the same argument as in \cref{lem:determine_D1}, except that here we use the bound  $4n-1\leq D_2\leq p\sqrt{8/D_1}$ from \cref{lem:succminbounds}; since $D_1\geq 15$ and $p\geq 11$  we can conclude $4n\leq p$ as before.
	\end{proof}

	Using our information about~$D_1$ and~$D_2$ we can determine the structure of a particular rank~$3$ sublattice of~$\mathcal{O}$. (The exact form of the Gram matrix is not important to the proof; we only require the fact that it can be determined knowing only~$D_1$,~$D_2$, and~$p$.)
	
	\begin{lem}\label{lem:rank3sublat}
		Let $\delta_i\in\{0,1\}$ satisfy $\delta_i\equiv D_i\pmod 2$ for $i=1,2$, and let~$T$ be the unique integer satisfying $0\leq T\leq \frac{p-1}{2}$ and $T^2\equiv D_1D_2\pmod p$. There exist $\alpha_1,\alpha_2\in\mathcal{O}$ such that $\tau(\alpha_1),\tau(\alpha_2)$ attain the first two successive minima for~$\mathcal{O}^T$, and the Gram matrix for $L\coloneqq \langle 1,\alpha_1,\alpha_2\rangle$ is 
		\[\begin{pmatrix}
			1 & \frac12\delta_1 & \frac12\delta_2\\
			\frac12\delta_1 & \frac14(D_1+\delta_1) & \frac14(T+\delta_1\delta_2) \\
			\frac12\delta_2 & \frac14(T+\delta_1\delta_2) &\frac14(D_2+\delta_2)
		\end{pmatrix}.\]
	\end{lem}
	\begin{proof}
		Let $\alpha_1,\alpha_2$ be such that 
		\[N(\tau(\alpha_i))=4N(\alpha_i)-\Tr(\alpha_i)^2=D_i\]
		for $i=1,2$. Adding an integer to~$\alpha_i$ if necessary, we may assume $\Tr(\alpha_i)=\delta_i$, so $N(\alpha_i)=\frac14(D_i+\delta_i)$. Replacing~$\alpha_2$ with~$\bar{\alpha_2}$ if necessary we can further assume that $\Tr(\tau(\alpha_1)\overline{\tau(\alpha_2)})\geq 0$.
		We have  $D_1\leq 2p^{2/3}\leq p$, so by \cref{lem:uniquetr} we have
		\[T=\frac12\Tr(\tau(\alpha_1)\overline{\tau(\alpha_2)})=2\Tr(\alpha_1\bar{\alpha_2})-\delta_1\delta_2.\]
		We can then solve for $\Tr(\alpha_1\bar{\alpha_2})= \frac12(T+\delta_1\delta_2)$, and this determines the Gram matrix for the basis $1,\alpha_1,\alpha_2$.
	\end{proof}

	\subsection{Determining~$D_3$ from~$D_1$ and~$D_2$}\label{sec:determining_D3}
	
	As above, we assume $D_1\geq 15$ (since \cref{lem:smallcases} applies when $D_1<15$).
	We can identify the fourth successive minimum of~$\mathcal{O}$ as the index of the smallest nonzero term of $\theta_{\mathcal{O}}-\theta_L$, where $L=\langle 1,\alpha_1,\alpha_2\rangle$ as in \cref{lem:rank3sublat}. Recall that in \cref{lem:determine_D1} (resp. \cref{lem:determine_D2}), we showed that if~$c_nq^n$ is the first nonzero term of $\theta_{\mathcal{O}}-\theta_\ZZ$ (resp. $\theta_{\mathcal{O}}-\theta_{\ZZ[\alpha_1]}$), then there are at most two elements with norm~$n$ in $\mathcal{O}\setminus\ZZ$ (resp. in $\mathcal{O}\setminus \ZZ[\alpha]$) up to negation and conjugation. Using this constraint we could determine the traces of these elements, and hence the first (resp. second) successive minimum of~$\mathcal{O}^T$.
	
	Unfortunately,~$n$ can be much larger than in previous cases, so \cref{lem:distinct_embeddings} may no longer be useful;  it is possible for there to exist two elements of~$\mathcal{O}$ of the same norm~$n$ generating distinct but isomorphic subfields. Thus the first nonzero coefficient may not be sufficient to determine~$D_3$. However, we will show that the first two coefficients of $\theta_{\mathcal{O}}-\theta_L$ are sufficient. 
	
	
	We begin with two plane geometry lemmas.

	\begin{lem}\label{lem:smallnorm}
		Let $\Lambda\subseteq\RR^2$ be a rank $2$ lattice with positive-definite quadratic form $Q$,  and $v\in\RR^2$. \black Let~$c$ denote the minimum of $Q(v-w)$ for $w\in\Lambda$, and~$\lambda$ the minimum of~$Q(w)$ for $w\in\Lambda\setminus\{0\}$. Then there exists a set of four points $P\subseteq\Lambda$, forming the vertices of a translated fundamental parallelogram of~$\Lambda$, such that for all $w\in\Lambda\setminus P$ we have $Q(v-w)\geq c+\lambda$.
	\end{lem}
	
	We call a set~$P$ satisfying the conclusion of \cref{lem:smallnorm} a \emph{separating set} for~$v$, since we can use it to ensure that no other element of $\Lambda$ is too close to $v$. Note that elements $w\in P$ may themselves satisfy $Q(v-w)\geq c+\lambda$, and a separating set for~$v$ in~$\Lambda$ is not necessarily unique.

	\black
	
	\begin{proof}
		
		Let $\cdot$ denote the bilinear form associated to $Q$. 
		By standard lattice basis reduction arguments, there exists a basis $u_1',u_2'$ for~$\Lambda$ such that the triangle with vertices $0,u_1',u_2'$ is right or acute, in the sense that $u_1'\cdot u_2'$, $(-u_2')\cdot (u_1'-u_2')$, and $(-u_1')\cdot (u_2'-u_1')$ are all non-negative. \black  The plane is tiled by congruent copies of this triangle with vertices lying in $\Lambda$; \black let~$\Delta$ be one such triangle containing~$v$ (allowing~$v$ to lie on the boundary of~$\Delta$). Let $r_1,r_2,r_3$ be the vertices of~$\Delta$, and~$s$ the orthocenter of~$\Delta$  (that is, $s$ satisfies $(r_i-s)\cdot (r_j-r_k)=0$ for all permutations $i,j,k$ of $1,2,3$). \black Then~$\Delta$ can be written as the union of three triangles~$\Delta_{12}$,~$\Delta_{23}$, and~$\Delta_{13}$, where~$\Delta_{ij}$ is defined as the triangle with vertices~$r_i,r_j,s$. Without loss of generality suppose $v\in \Delta_{12}$. After translation by~$-r_3$, and setting $u_1\coloneqq r_1-r_3$ and~$u_2 \coloneqq r_2-r_3$, we can assume that~$\Delta$ has vertices~$0,u_1,u_2$ and that~$v$ lies in the triangle with vertices~$s,u_1,u_2$, as in \cref{fig:triangle}. By computing the orthogonal projections of~$u_2,v,u_1$ onto the span of~$u_1$ (and similarly onto the span of~$u_2$), we obtain the relations
		\begin{align*}
			0\leq u_2\cdot u_1 \leq v\cdot u_1 \leq  u_1\cdot u_1,\\
			0\leq  u_1\cdot u_2 \leq  v\cdot u_2 \leq  u_2\cdot u_2.
		\end{align*}
		To simplify notation set 
		\[ t\coloneqq u_1\cdot u_2,\qquad s_1\coloneqq  v\cdot u_1,\qquad s_2\coloneqq  v\cdot u_2,\]
		so we have $0\leq t\leq s_i\leq Q(u_i)$ for each $i=1,2$.

		\begin{figure}[h]
			\begin{center}
				\begin{tikzpicture}
					\filldraw[gray!20] (4,0) -- (1,0.6) -- (1,5) -- (4,0); 
					\draw  (0,0) edge (1,5);
					\draw  (0,0) edge (4,0);
					\draw  (1,5) edge (4,0);
					\draw  (1,5) edge (5,5);
					\draw  (4,0) edge (5,5);
					\draw[dashed]  (1,5) edge (1,0);
					\draw[dashed]  (4,0) edge (0.153846,0.76923);
					\filldraw[black] (0,0) circle (2pt) node[anchor=north]{$0$};
					\filldraw[black] (4,0) circle (2pt) node[anchor=north]{$u_1$};
					\filldraw[black] (1,5) circle (2pt) node[anchor=south]{$u_2$};
					\filldraw[black] (5,5) circle (2pt) node[anchor=south]{$u_1+u_2$};
					\filldraw[black] (2,2) circle (2pt) node[anchor=west]{$v$};
					\filldraw[black] (1,0.6) circle (2pt) node[anchor=215]{$s$};
					\draw[dashed]  (2,2) edge (2,0);
					\draw[dashed]  (2,2) edge (0.461538,2.307692);
				\end{tikzpicture}
			\end{center}
			\caption{If $v$ lies in the highlighted gray triangle, then $P=\{0,u_1,u_2,u_1+u_2\}$ satisfies the conclusion of \cref{lem:smallnorm}.}\label{fig:triangle}
		\end{figure}
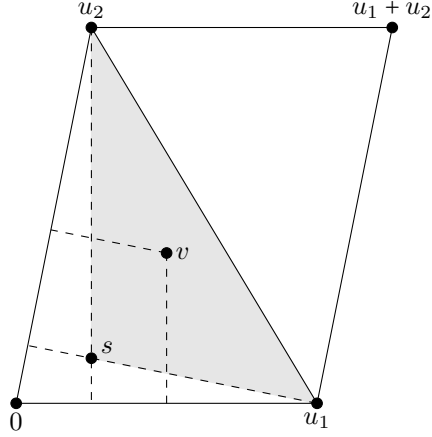
		
		Now for any lattice element $w\in\Lambda$, we exhibit $w_0\in P\coloneqq \{0,u_1,u_2,u_1+u_2\}$ such that $(v-w_0)\cdot (w-w_0)\leq 0$. Set $w=as_1+bs_2$ for some $a,b\in\ZZ$. 
		\begin{itemize}
			\item If $a,b\leq 0$ then
			$v\cdot w=as_1+bs_2 \leq 0$.
			\item If $a\geq 1$ and $b\leq 0$ then $(v-u_1)\cdot(w-u_1)=(a-1)(s_1-Q(u_1))+b(s_2-t) \leq 0$.
			\item If $a\leq 0$ and $b\geq 1$ then
			$(v-u_2)\cdot(w-u_2)=a(s_1-t)+(b-1)(s_2-Q(u_2)) \leq 0$.
			\item If $a,b\geq 1$ then
			\[(v-u_1-u_2)\cdot(w-u_1-u_2)=(a-1)(s_1-Q(u_1)-t)+(b-1)(s_2-Q(u_2)-t) \leq 0.\]
		\end{itemize}
		We therefore have
		\begin{align*}
			Q(v-w)&\geq Q(v-w_0)+Q(w-w_0).
		\end{align*}
		We have $Q(v-w_0)\geq c$, and $Q(w-w_0)\geq\lambda$ unless $w=w_0$.		
	\end{proof}

	\begin{lem}\label{lem:R2traces}
		Let $b_1,b_2\in \RR^2$ be linearly independent vectors, and $v_1,v_2\in\RR^2$ distinct vectors. Let $(x,y)\mapsto x\cdot y$ denote the bilinear form associated to a positive-definite quadratic form on $\RR^2$. If $b_1\cdot b_2\neq 0$, $v_1\cdot v_1=v_2\cdot v_2$, and $|v_1\cdot b_j|=|v_2\cdot b_j|$ for $j=1,2$, then $v_1=-v_2$.
	\end{lem}
	See \cref{fig:R2traces} for an intuitive explanation of this result.
	\black
	\begin{figure}[h]
		\begin{center}
			\begin{tikzpicture}
				
				\draw  (-.3,-1.5) edge (.5,2.5);
				\draw  (-1.5,0) edge (2,0);
				\filldraw[black] (0,0) circle (2pt) node[anchor=110]{$0$};
				\filldraw[black] (2,0) circle (2pt) node[anchor=north]{$b_1$};
				\filldraw[black] (.5,2.5) circle (2pt) node[anchor=south]{$b_2$};
				\filldraw[black] (1,1) circle (2pt) node[anchor=west]{$v$};
				\draw[dashed]  (1,1) edge (-1,1.4);
				\draw[dashed]  (-1,1.4) edge (-1,-1);
				\draw[dashed]  (-1,-1) edge (1,-1.4);
				\draw[dashed]  (1,1) edge (1,-1.4);
				\draw[fill=white] (-1,1.4) circle (3pt);
				\draw[fill=white] (-1,-1) circle (3pt);
				\draw[fill=white] (1,-1.4) circle (3pt);
			\end{tikzpicture} 
		\end{center}
		\caption{The white circles indicate the points $w\in\RR^2$ with $w\neq v$ and $|w\cdot b_j|=|v\cdot b_j|$ for $j=1,2$. If $b_1$ and $b_2$ are not orthogonal, then the only such point with the same norm as $v$ is $-v$.}\label{fig:R2traces}
	\end{figure}
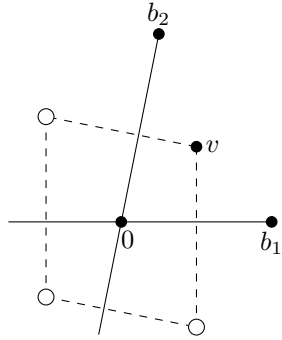
	
	\begin{proof}
		Since $b_1,b_2$ form a basis for $\RR^2$, it suffices to show that $v_1\cdot b_j=-v_2\cdot b_j$ for each $j=1,2$. For the sake of contradiction, suppose $v_1\cdot b_1=v_2\cdot b_1\neq 0$. Since $v_1$ and $v_2$ are distinct, we must then have $v_1\cdot b_2\neq v_2\cdot b_2$, so $v_1\cdot b_2=-v_2\cdot b_2$. Set $b_2^*=b_2-\frac{b_2\cdot b_1}{b_1\cdot b_1}b_1$, so that for $i=1,2$ we can write
		\[v_i=\frac{v_i\cdot b_1}{b_1\cdot b_1}b_1+\frac{v_i\cdot b_2^*}{b_2^*\cdot b_2^*}b_2^*\]
		with $b_1,b_2^*$ orthogonal. Since $v_1\cdot v_1=v_2\cdot v_2$ but $v_1\neq v_2$, we can conclude that $v_1\cdot b_2^*=-v_2\cdot b_2^*$, or expanding,
		\[v_1\cdot b_2-\frac{b_2\cdot b_1}{b_1\cdot b_1}(v_1\cdot b_1)=-v_2\cdot b_2+\frac{b_2\cdot b_1}{b_1\cdot b_1}(v_2\cdot b_1).\]
		Since $v_1\cdot b_2=-v_2\cdot b_2$ and $v_1\cdot b_1=v_2\cdot b_1\neq 0$, we conclude $b_2\cdot b_1=0$, a contradiction. Hence $v_1\cdot b_1=-v_2\cdot b_1$, and similarly $v_1\cdot b_2=-v_2\cdot b_2$, so that $v_1=-v_2$.
	\end{proof}
	
	\black
	
	We now return to the quaternion setting; we continue to assume $D_1\geq 15$ and $p\geq 11$ (though all we will use from now on is $D_1>5$ and $p$ odd). Note that from $L=\langle 1,\alpha_1,\alpha_2\rangle$ we obtain a rank~$2$ lattice 
	\[\tau(L)=\langle\beta_1,\beta_2\rangle;\]
	under the  isomorphism between $B_p\otimes\RR$ and~$\RR^4$, we obtain a lattice in~$\RR^2$ to which we can apply \cref{lem:smallnorm} and \cref{lem:R2traces}.
	
	\begin{lem}\label{lem:fourpairs}
		There exists a set $S\subseteq \mathcal{O}^T\setminus\tau(L)$ of four elements with the following properties:
		\begin{enumerate}[label=(\alph*)]
			\item For all $\gamma\in \mathcal{O}^T\setminus (\tau(L)\cup S\cup -S)$ we have $N(\gamma)\geq D_3+D_1$.
			\item At most two elements of~$S$ have even norm.
			\item If $S$ contains two elements $\gamma_1,\gamma_2$ with $N(\gamma_1)=N(\gamma_2)<D_3+D_1$, then the remaining two elements $\gamma_3,\gamma_4\in S$ satisfy $N(\gamma_3)=N(\gamma_4)$.
		\end{enumerate}
	\end{lem}

	\begin{figure}[h]
		\begin{center}
			\begin{tikzpicture}
				\pgfmathsetmacro{\bx}{2};
				\pgfmathsetmacro{\by}{0};
				\pgfmathsetmacro{\cx}{0.5};
				\pgfmathsetmacro{\cy}{0.4};
				\pgfmathsetmacro{\vx}{0.7};
				\pgfmathsetmacro{\vy}{0.2};
				\pgfmathsetmacro{\uy}{2.3};
				
				\pgfmathsetmacro{\Tblx}{-1.2*\bx-2*\cx};
				\pgfmathsetmacro{\Tbly}{-1.2*\by-2*\cy};
				\pgfmathsetmacro{\Tbrx}{-1.2*\bx+2*\cx};
				\pgfmathsetmacro{\Tbry}{-1.2*\by+2*\cy};
				\pgfmathsetmacro{\Ttlx}{2*\bx-2*\cx};
				\pgfmathsetmacro{\Ttly}{2*\by-2*\cy};
				\pgfmathsetmacro{\Ttrx}{2*\bx+2*\cx};
				\pgfmathsetmacro{\Ttry}{2*\by+2*\cy};
				
				\foreach \k in {-1,...,1}
				{
					\draw  (\Tblx,\k*\uy+\Tbly) edge (\Ttlx, \k*\uy+\Ttly);
					\draw  (\Ttrx,\k*\uy+\Ttry) edge (\Ttlx, \k*\uy+\Ttly);
					\draw  (\Tblx,\k*\uy+\Tbly) edge (\Tbrx, \k*\uy+\Tbry);
					\draw  (\Tbrx,\k*\uy+\Tbry) edge (\Ttrx,\k*\uy+ \Ttry);
				}
				
				\filldraw[gray!20]  (\vx,\vy+\uy) -- (\vx-\bx,\vy+\uy-\by) -- (\vx-\bx-\cx,\vy+\uy-\by-\cy) -- (\vx-\cx,\vy+\uy-\cy) -- (\vx,\vy+\uy);
				\filldraw[gray!20]  (-\vx,-\vy-\uy) -- (-\vx+\bx,-\vy-\uy+\by) -- (-\vx+\bx+\cx,-\vy-\uy+\by+\cy) -- (-\vx+\cx,-\vy-\uy+\cy) -- (-\vx,-\vy-\uy);
				\filldraw[gray!20]  (0,0) -- (\bx,\by) -- (\bx+\cx,\by+\cy) -- (\cx,\cy) -- (0,0);
				
				\draw[dashed]  (\vx,\vy) edge (\vx, \vy+\uy);
				
				\draw[fill=black] (0,0) circle (2pt) node[anchor=90]{$0$};
				\draw[fill=black] (\bx,\by) circle (2pt);
				\draw[fill=black] (\cx,\cy) circle (2pt);
				\draw[fill=black] (\bx+\cx,\by+\cy) circle (2pt);
				\draw[fill=white] (\vx,\vy) circle (2pt) node[anchor=190]{$v$};
				\node[anchor=190] at (\vx,\vy/2+\uy/2-.2) {$u$};
				\node at (\bx/2+\cx/2+0.3,\by/2+\cy/2) {$P$};

				\draw[fill=black] (\vx,\vy+\uy) circle (2pt) node[anchor=180]{$\beta_3$};
				\draw[fill=black] (\vx-\bx,\vy+\uy-\by) circle (2pt);
				\draw[fill=black] (\vx-\cx,\vy+\uy-\cy) circle (2pt);
				\draw[fill=black] (\vx-\bx-\cx,\vy+\uy-\by-\cy) circle (2pt);
				\node at (\vx-\bx/2-\cx/2,\vy+\uy-\by/2-\cy/2) {$S$};

				\draw[fill=black] (-\vx,-\vy-\uy) circle (2pt) node[anchor=0]{$-\beta_3$};
				\draw[fill=black] (-\vx+\bx,-\vy-\uy+\by) circle (2pt);
				\draw[fill=black] (-\vx+\cx,-\vy-\uy+\cy) circle (2pt);
				\draw[fill=black] (-\vx+\bx+\cx,-\vy-\uy+\by+\cy) circle (2pt);
				\node at (-\vx+\bx/2+\cx/2,-\vy-\uy+\by/2+\cy/2) {$-S$};
				
				\node at (3.5,0.5) {$\tau(L)$};
				\node at (3.5,\uy+0.5) {$\beta_3+\tau(L)$};
				\node at (3.5, -\uy+0.5) {$-\beta_3+\tau(L)$};
			\end{tikzpicture}
		\end{center}
		\caption{Setup for the proof of \cref{lem:fourpairs}. Black dots correspond to elements of $\mathcal{O}^T$. The sets $P$, $S$, and $-S$ are the vertices of the parallelograms with the corresponding labels.}\label{fig:3dpic}
	\end{figure}
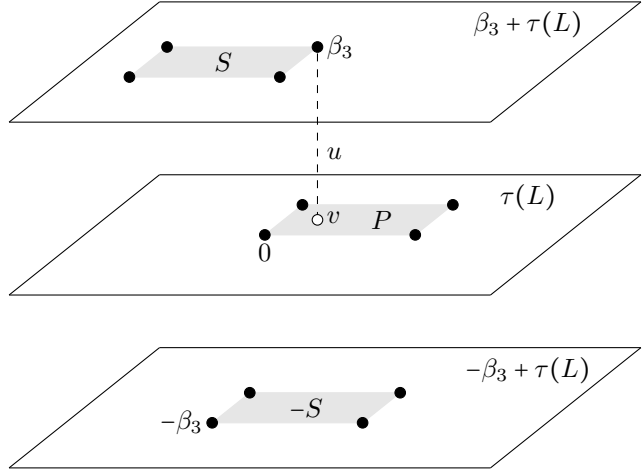
	
	\begin{proof}
		Let $\beta_1,\beta_2,\beta_3\in\mathcal{O}^T$ attain the successive minima $D_1,D_2,D_3$. By \cref{lem:succminbasis}, these elements form a basis of~$\mathcal{O}^T$, so the quotient~$\mathcal{O}^T/\tau(L)$ is  an infinite cyclic group generated by $\beta_3$. \black Let~$v$ be the orthogonal projection of~$\beta_3$ onto $\RR\tau(L)$, so $u\coloneqq\beta_3-v$ is orthogonal to~$\RR\tau(L)$.  Note that if we had $N(v-w)<N(v)$ for some $w\in \tau(L)$, this would imply
		\[N(\beta_3-w)=N(v-w)+N(u)<N(v)+N(u)= N(\beta_3),\]
		contradicting the fact that~$\beta_3$ attains the third successive minimum. Hence $N(v)\leq N(v-w)$ for all $w\in\tau(L)$. In particular,~$N(v)$ is bounded by the square of the covering radius of $\tau(L)$; since the covering radius is bounded by $\frac{\sqrt{2}}{2}\sqrt{D_2}$ we have  $N(v)\leq\frac12 D_2$.
		
		Now any $\gamma\in\mathcal{O}^T\setminus \tau(L)$ can be written in the form
		\[\gamma=a\beta_3-w=au+(av-w)\]
		for some $a\in\ZZ\setminus \{0\}$ and $w\in\tau(L)$. If $|a|\geq 2$ then using $N(v)\leq\frac12D_2\leq \frac12D_3$ we have
		\[N(\gamma)=|a|^2N(u)+N(av-w)\geq 4N(u)= 4(N(\beta_3)-N(v))\geq 2D_3\geq D_3+D_1.\]
		On the other hand, suppose $a=1$, so that $\gamma=u+(v-w)$. Then by \cref{lem:smallnorm}, there exists  $P\subseteq \tau(L)$ forming the vertices of a translated fundamental parallelogram for~$\tau(L)$ such that for all $w\in\tau(L)\setminus P$ \black we have
		\[N(\gamma)=N(u)+N(v-w)\geq N(u)+N(v)+D_1=D_3+D_1.\]
		In other words, we have $N(\gamma)\geq D_3+D_1$ provided $\gamma\notin \beta_3-P$. Finally, if $a=-1$ (so $\gamma=-u+(-v-w)$), the same argument shows that $N(\gamma)\geq D_3+D_1$ unless $\gamma\in -\beta_3+P$. \black So taking $S\coloneqq\beta_3-P$ as in \cref{fig:3dpic}, (a) follows.
		
		Now for the sake of contradiction suppose~$N(\gamma_i)$ is even for three elements $\gamma_1,\gamma_2,\gamma_3\in S$. The points $w_i=\beta_3-\gamma_i$ lie on~$P$, so their pairwise differences $\gamma_i-\gamma_j=w_j-w_i$ (for $i,j\in\{1,2,3\}$) contain a basis for~$\tau(L)$. This implies that $\langle \gamma_1,\gamma_2,\gamma_3\rangle=\mathcal{O}^T$. But since~$\Tr(uv)$ is even for all $u,v\in\mathcal{O}^T$  (\cref{eq:eventrace}), \black the set of points in~$\mathcal{O}^T$ with even norm is closed under addition. Hence every element of~$\mathcal{O}^T$ must have even norm. This is a contradiction, because~$\mathcal{O}$ contains an element of odd trace (see \cref{eq:normtau} and the discussion after \cref{eq:gross_containments}). Thus (b) must hold.
		
		Finally, suppose $N(\gamma_1)=N(\gamma_2)<D_3+D_1$ for some distinct $\gamma_1,\gamma_2\in S$. For each $i,j\in\{1,2\}$ we have 
		\[D_3\leq N(\gamma_i\pm \beta_j)=N(\gamma_i)+N(\beta_j)\pm \Tr(\gamma_i\bar{\beta_j})<(D_3+D_1)+D_2\pm \Tr(\gamma_i\bar{\beta_j}),\]
		and, since $D_1,D_2<p$ by \cref{lem:succminbounds}  (recall we are assuming $D_1\geq 15$),  we have
		$|\Tr(\gamma_i\bar{\beta_j})|<2p$. We also have $|\frac{1}{2}\Tr(\gamma_i\bar{\beta_j})|^2\equiv N(\gamma_i)N(\beta_j)\pmod{4p}$ by \cref{prop:kaneko}. So $|\frac{1}{2}\Tr(\gamma_i\bar{\beta_j})|$ equals the square root modulo~$2p$ of $N(\gamma_i)N(\beta_j)$ in the interval~$[0,p)$, which is unique because~$(\ZZ/2p\ZZ)^\times$ is cyclic.

		For $i=1,2$ let $v_i$ be the projection of $\gamma_i$ onto $\RR\tau(L)$, so that $\gamma_i=u+v_i$ with $u$ orthogonal to $v_i$; note that $v_1\neq v_2$. Since $N(\gamma_1)=N(\gamma_2)$ we have $N(v_1)=N(v_2)$, and for $j=1,2$ we have
		\[|\Tr(v_1\bar{\beta_j})|=|\Tr(\gamma_1\bar{\beta_j})|=|\Tr(\gamma_2\bar{\beta_j})|=|\Tr(v_2\bar{\beta_j})|.\]
		Further, we have $0<N(\beta_1),N(\beta_2)<p$ by \cref{lem:succminbounds}, and so $\Tr(\beta_2\bar{\beta_1})\neq 0$ by \cref{prop:kaneko}. So by \cref{lem:R2traces}, we can conclude $v_1=-v_2$. Now if $\gamma_3,\gamma_4\in S$ are the remaining two elements with $\gamma_3=u+v_3$ and $\gamma_4=u+v_4$, then we must have $v_3=-v_4$ because the four elements of~$S$ form the vertices of a parallelogram. Hence $N(\gamma_3)=N(\gamma_4)$, proving (c).
	\end{proof}
	
	\begin{lem}
		Suppose $D_1\geq 15$, and let 
		\[(\theta_\mathcal{O}-\theta_L)(q)=c_nq^n+c_{n+1}q^{n+1}+\cdots\]
		for some $n\geq 0$. If $c_n=2$, or if $c_n=4$ and $c_{n+1}\geq 8$, then $D_3=4n$; otherwise $D_3=4n-1$.
	\end{lem}
	
	\begin{proof}
		As in \cref{lem:determine_D1}, the existence of an element of norm~$n$ in $\mathcal{O}\setminus L$  (and no smaller norm)  guarantees that~$D_3$ is equal to either~$4n-1$ or~$4n$. Let $S\subseteq \mathcal{O}^T\setminus\tau(L)$ be as in \cref{lem:fourpairs}. Since $N(\gamma)\geq D_3+15>4n+4$ for all $\gamma\in\mathcal{O}^T\setminus(\tau(L)\cup S\cup -S)$, the series $\theta_{\mathcal{O}}(q)-\theta_{L}(q)$ is congruent to
		\begin{equation}\label{eq:thetadiff}
			\left(\sum_{\substack{\beta\in S\\N(\beta)\equiv 0\;\mathrm{mod}\;4}} 2q^{N(\beta)/4}\right)\theta_0(q)+\left(\sum_{\substack{\beta\in S\\N(\beta)\equiv 3\;\mathrm{mod}\;4}} 2q^{(1+N(\beta))/4}\right)\theta_1(q)\pmod{q^{n+2}}.
		\end{equation}
		We divide into cases based on the number of elements of norm~$4n-1$ and~$4n$ in~$S$. First we consider the cases that~$S$ has no elements of norm~$4n-1$, so that $D_3=4n$. Then~$S$ contains either one or two elements of norm~$4n$:~$S$ cannot contain  more than two  elements of norm~$4n$ by \cref{lem:fourpairs}(b). Recall that $\theta_0(q)\equiv 1+2q\pmod{q^2}$ and $\theta_1(q)\equiv 2\pmod{q^2}$.
		\begin{itemize}
			\item If~$S$ contains one element of norm~$4n$, this element contributes~$2q^n\theta_0(q)$ to \cref{eq:thetadiff} and every other element of~$S$ contributes a multiple of~$q^{n+1}$, so $c_n=2$.
			\item If~$S$ contains two elements of norm~$4n$, these elements contribute~$4q^n\theta_0(q)$ to \cref{eq:thetadiff} and every other element of~$S$ contributes a multiple of~$q^{n+1}$, so $c_n=4$ and $c_{n+1}\geq 8$.
		\end{itemize}
		Now consider the cases that~$S$ has at least one element of norm~$4n-1$, so that~$D_3=4n-1$.
		\begin{itemize}
			\item If~$S$ contains at least two elements of norm~$4n-1$, these contribute~$4q^n\theta_1(q)$ to \cref{eq:thetadiff}, so $c_n\geq 8$.
			\item If~$S$ contains at least one element of norm~$4n-1$ and at least one element of norm~$4n$, these contribute~$2q^n\theta_0(q)+2q^n\theta_1(q)$ to \cref{eq:thetadiff}, so $c_n\geq 6$.
			\item Suppose~$S$ contains exactly one element of norm~$4n-1$, and no elements of norm~$4n$. The element of norm~$4n-1$ contributes~$2q^n\theta_1(q)$ to \cref{eq:thetadiff}, so $c_n=4$. Now by \cref{lem:fourpairs}(c),~$S$ contains at most one element of norm~$4n+3$ and at most one element of norm~$4n+4$. Thus the largest possible value of~$c_{n+1}$ is attained by 
			\[2q^n\theta_1(q)+2q^{n+1}\theta_0(q)+2q^{n+1}\theta_1(q),\]
			which has $c_{n+1}=6$. In any case we will have $c_n=4$ and $c_{n+1}\leq 6$.
		\end{itemize}
		Combining these results, we see that if $D_3=4n$ then either we have $c_n=2$ or we have $c_n=4$ and $c_{n+1}\geq 8$. When $D_3=4n-1$, then either we have $c_n\geq 6$ or we have $c_n=4$ and $c_{n+1}\leq 6$.
	\end{proof}

	In conclusion, the theta function of~$\mathcal{O}$ uniquely determines the successive minima $D_1,D_2,D_3$ of~$\mathcal{O}^T$. By \cref{thm:succmin_to_iso}, these values uniquely determine the isomorphism type of~$\mathcal{O}$, establishing \cref{thm:theta_to_iso}.

				\printbibliography

			\end{document}